\documentclass[10pt,a4paper]{article}

\author{Ofir Gorodetsky} \title{Irreducible polynomials over $\FF_{2^r}$ with three prescribed coefficients}
\date{}

\usepackage[margin=1in]{geometry}
\usepackage[all]{xy}
\usepackage{amssymb}
\usepackage{amsfonts}
\usepackage{amsthm}
\usepackage{amsmath}
\usepackage{hyperref}
\usepackage{mathtools}
\usepackage{url}

\newtheorem*{thm*}{Theorem}
\newtheorem{thm}{Theorem}[section]

\newtheorem{lem}[thm]{Lemma}  
\newtheorem{proposition}[thm]{Proposition}
\newtheorem{cor}[thm]{Corollary}
\theoremstyle{definition}

\newtheorem{remark}[thm]{Remark}

\newcommand{\QQ}{\mathbb{Q}}
\newcommand{\ZZ}{\mathbb{Z}}
\newcommand{\FF}{\mathbb{F}}

\numberwithin{equation}{section}

\mathtoolsset{showonlyrefs,showmanualtags}

\begin{document}
\maketitle
\begin{abstract}
	For any positive integers $n \ge 3$ and $r \ge 1$, we prove that the number of monic irreducible polynomials of degree $n$ over $\FF_{2^r}$ in which the coefficients of $T^{n-1}$, $T^{n-2}$ and $T^{n-3}$ are prescribed has period $24$ as a function of $n$, after a suitable normalization. A similar result holds over $\FF_{5^r}$, with the period being $60$. We also show that this is a phenomena unique to characteristics $2$ and $5$. The result is strongly related to the supersingularity of certain curves associated with cyclotomic function fields, and in particular it complements an equidistribution result of Katz.
\end{abstract}

\section{Introduction}
Let $q=p^r$ be a prime power and let $\FF_q$ denote the finite field containing $q$ elements. Let $\FF_q[T]$ denote the polynomial ring over $\FF_q$ in indeterminate $T$. Let $\mathcal{M}_{n,q}$ denote the set of monic polynomials of degree $n$ in $\FF_q[T]$ and $\mathcal{M}_q = \cup_{n \ge 0} \mathcal{M}_{n,q}$ denote the set of all monic polynomials in $\FF_q[T]$. Let $\mathcal{P}_q$ denote the set of irreducible polynomials in $\mathcal{M}_{q}$. We adopt throughout the following conventions: The $j$-th next-to-leading coefficient of a polynomial $f(T) \in \mathcal{M}_q$ with $j > \deg f$ is considered to be $0$. We assume that $0 \in \mathbb{N}$.

A classical result due to Gauss \cite[pp.~608-611]{gauss1965} is that the number $\pi_q(n)$ of polynomials of degree $n \ge 1$ in $\mathcal{P}_q$ is given by
\begin{equation}\label{gaussprimecount}
\pi_q(n) = \frac{1}{n}\sum_{d \mid n} \mu(d) q^{n/d},
\end{equation}
where $\mu$ is the usual M\"obius function.
A natural extension of this problem is to determine the number of monic irreducible polynomials in $\FF_q[T]$ of degree $n$ for which the first $\ell$ next-to-leading coefficients have the prescribed values $t_1, \ldots, t_{\ell}$, while the remaining coefficients are arbitrary.
Formally, if we let $1_{\mathcal{P}_q} \colon \mathcal{M}_q \to \mathbb{C}$ be the indicator of primes in $\mathcal{M}_q$, then one may study the function
\begin{equation}
\pi_q(n,t_1,\ldots,t_{\ell}) = \sum_{ \substack{f \in \mathcal{M}_{n,q}, \\ \text{ first }\ell \text{ next-to-leading} \\\text{coefficients of } f\text{ are }t_1,\ldots, t_{\ell}}}  1_{\mathcal{P}_q}(f).
\end{equation}
We note that $\pi_q(n,t_1,\ldots, t_{\ell})$ makes sense for any $\ell \in \mathbb{N}$, including $\ell=0$ and $\ell>n$. Part of the motivation for studying  $\pi_q(n,t_1,\ldots, t_{\ell})$ is the following observation. 
If we let 
\begin{equation}
f(T)= T^n + t_1 T^{n-1} + \ldots + t_{\ell} T^{n-\ell},
\end{equation}
then $\pi_q(n,t_1,\ldots,t_{\ell})$ counts the number of irreducible polynomials in the set
\begin{equation}
\{ g \in \FF_q[T] : \deg (g - f) < n-\ell \}, 
\end{equation}
which is a polynomial analogue of an interval in $\mathbb{Z}$, see the discussion in Keating and Rudnick \cite[Sec.~2.1]{keating2014}. The problem of counting integer primes in an interval has a long history. If we let $\pi(x) = \# \{0<p \le x : p\text{ is a prime}\}$ be the prime counting function, then the prime number theorem states that
\begin{equation}
\pi(x) \sim \frac{x}{\ln x}, \qquad x \to \infty.
\end{equation}
Therefore, one may expect that if $\Phi(x)$ is a function that tends to infinity slower than $x$ but not too slow, then the `short' interval $[x,x+\Phi(x)]$ contains about $\frac{\Phi(x)}{\ln x}$ primes, that is,
\begin{equation}\label{estphiint}
\pi(x+\Phi(x))-\pi(x)\sim \frac{\Phi(x)}{\ln x}, \qquad x \to \infty.
\end{equation}
It is conjectured \cite[p.~7]{granville1995} that \eqref{estphiint} holds for $\Phi(x)=x^{\varepsilon}$ for any $\varepsilon \in (0,1)$. Von Koch \cite{koch01} showed that the Riemann Hypothesis implies that \eqref{estphiint} holds for $\Phi(x)=\sqrt{x}\ln^2 x$. Heath-Brown \cite{heathbrown1988}, improving on Huxley \cite{huxley1972}, proved that \eqref{estphiint} holds unconditionally for $\Phi(x)=x^{\frac{7}{12}-\varepsilon(x)}$, where $\varepsilon(x) \to 0$ as $x \to \infty$.

For $\ell < \frac{n}{2}$, methods similar to the ones in the integer setting yield good estimates for $\pi_q(n,t_1,\ldots,t_{\ell})$. In the polynomial setting the Riemann Hypothesis for $L$-Functions was established by Weil \cite{weil1974}, from which one can deduce
\begin{equation}\label{eq:piqest}
\pi_q(n,t_1,\ldots,t_{\ell})= \frac{q^{n-\ell}}{n} +O\left( \frac{\ell+1}{n} q^{\frac{n}{2}}\right),
\end{equation}
see Rhin \cite[Thm.~4]{rhin1972} (cf. Hsu \cite[Cor.~2.6]{hsu1996}). For $n-4 \ge \ell \ge \frac{n}{2}$, asymptotic results are known only in the limit $q \to \infty$, and the methods come from Galois theory. These methods give
\begin{equation}
\pi_q(n,t_1,\ldots,t_{\ell}) = \frac{q^{n-\ell}}{n} + O_n\left(q^{n-\ell-\frac{1}{2}}\right),
\end{equation} 
see the works of Cohen \cite[Thm.~1]{cohen1972} and Bank, Bary-Soroker and Rosenzweig \cite[Thm.~2.3]{bank2015}.

Apart from asymptotic results, several exact results are known. Carlitz obtained formulas for $\pi_q(n,t_1)$ in 1952 \cite[Thm.~3]{carlitz1952}, and Kuz'min obtained formulas for $\pi_q(n, t_1, t_2)$ in 1991 \cite[Thms.~1-4]{kuzmin1991}.

Both Carlitz and Kuz'min have used extensively in their proofs $L$-functions of certain characters. Later, several authors have reproved the results of Carlitz and Kuz'min without employing $L$-functions. Ruskey \emph{et al.} \cite[Thm.~1.1]{ruskey2001} and Yucas \cite[Cor.~2.7]{yucas2006} reproduced Carlitz' result. Cattell \emph{et al.} reproduced Kuz'min's result for the case $q=2$ \cite{cattell2003}, and Ri \emph{et al.} reproduced Kuz'min's result for the case that $q$ is a power of 2 \cite[Thms.~1-4]{ri2014}. Recently, Lal\'\i n and Larocque \cite{lalin2016} obtained a combinatorial proof of Kuz'min's result.

Apart from Carlitz's and Kuz'min's work on $\pi_q(n,t_1,t_2)$, some other interesting exact results are known. Kuz'min determined $\pi_q(4,t_1,t_2,t_3)$ for any $q$ \cite{kuzmin1994}. Yucas and Mullen computed $\pi_2(n, t_1, t_2, t_3)$ for $n$ even \cite[Thm.~4]{yucas2004}, and Fitzgerald and Yucas computed $\pi_2(n, t_1, t_2, t_3)$ for $n$ odd \cite[Thm.~3.4]{fitzgerald2003}. Moisio and Ranto \cite[Lem.~19]{moisio2008} computed $\sum_{t_2 \in \FF_q} \pi_{2^r}(n,0,t_2,t_3)$ for any given $t_3 \in \FF_{2^r}$. Recently, Ahmadi \emph{et al.} computed $\pi_{2^r}(n,0,0,0)$ for any $n$ \cite[Thm.~8]{ahmadi2016}. The mentioned papers again did not employ $L$-functions.
\subsection{Counting with von Mangoldt}
Let $\Lambda_q \colon \mathcal{M}_q \to \mathbb{C}$ be the von Mangoldt function, defined as
\begin{equation}
\Lambda_q(f) = \begin{cases} \deg P & \mbox{if $f=P^k$, $P$  irreducible,} \\ 0 & \mbox{otherwise.} \end{cases}
\end{equation}
The von Mangoldt function counts prime powers with weights, and is closely related to the indicator of primes, $1_{\mathcal{P}_q}$. Let $\psi_q\colon \mathbb{N} \to \mathbb{C}$ be the summatory function of $\Lambda_q$:
\begin{equation}
\psi_q(n) = \sum_{f \in \mathcal{M}_{n,q}} \Lambda_q(f).
\end{equation}
We have the following classical result for any positive integer $n$, from which \eqref{gaussprimecount} is usually derived \cite[Prop.~2.1]{rosen2002}:
\begin{equation}\label{psicount}
\psi_q(n) =q^n.
\end{equation}
Given $t_1,\ldots,t_{\ell} \in \FF_q$, we denote by $\psi_q(n,t_1,\ldots,t_{\ell})$ the sum of $\Lambda_q$ over elements of $\mathcal{M}_{n,q}$ whose first $\ell$ next-to-leading coefficients are $t_1, \ldots, t_{\ell}$:
\begin{equation}
\psi_q(n,t_1,\ldots,t_{\ell}) = \sum_{ \substack{f \in \mathcal{M}_{n,q}, \\ \text{ first }\ell \text{ next-to-leading } \\ \text{coefficients of } f\text{ are }t_1,\ldots, t_{\ell}}}  \Lambda_q(f).
\end{equation}
A standard M\"obius inversion argument allows one to switch between $\psi_q(n,\bullet,\ldots,\bullet)$ and $\pi_q(n,\bullet, \ldots, \bullet)$. This is done explicitly in Proposition \ref{fullmobinv} for any $\ell$ and in Corollary \ref{inv32} for $\ell=3$. The estimate \eqref{eq:piqest} takes the following form in the language of $\psi_q$:
\begin{equation}\label{eq:psiest}
\psi_q(n,t_1,\ldots,t_{\ell})= q^{n-\ell}+O\left( (\ell+1) q^{\frac{n}{2}}\right).
\end{equation}
\subsection{Results}\label{ressec}
In \cite[Thm.~8]{ahmadi2016}, the value of the function $\psi_q(n,0,0,0)$ was determined for any $q$ which is a power of $2$. Let
\begin{equation}
f_q(n,t_1,t_2,t_3) = \frac{\psi_q(n,t_1,t_2,t_3) - q^{n-3}}{q^{n/2}}.
\end{equation}
The authors of \cite{ahmadi2016} proved that $f_q(n,0,0,0)$ has period $24$ in $n$ starting from $n=3$, that is,
\begin{equation}
f_{2^r}(n+24,0,0,0) = f_{2^r}(n,0,0,0)
\end{equation}
for all $r \ge 1$ and $n \ge 3$. Ahmadi \emph{et al.} write \cite[Sec.~9]{ahmadi2016}: ``Based on the $\FF_2$ base field case and our result for $f_q(n,0,0,0)$, we conjecture that for $q = 2^r$, $n \ge 3$ and $t_1, t_2, t_3 \in \FF_q$, the formulas for $f_q(n,t_1,t_3,t_3)$ all have period $24$". We prove this conjecture and exhibit some additional symmetries.
\begin{thm}\label{24period}
Let $q=2^r$. Let $t_1,t_2,t_3 \in \FF_q$. Then 
\begin{equation}
f_q(n,t_1,t_2,t_3)=\frac{\psi_q(n,t_1,t_2,t_3)- q^{n-3}}{q^{n/2}} \text{ is periodic in }n \in \mathbb{N}_{>0} \text{ with period }24.
\end{equation}
Additionally, the following symmetries hold. Let $n \ge 1$. If $\lambda\in \mathbb{N}_{>0}$ is coprime to $6$, then
\begin{equation}
f_q(n,t_1,t_2,t_3) = f_q\left(\lambda  n, t_1,  t_2 + \binom{\lambda}{2} t_1^2, t_3 + \binom{\lambda}{3} t_1^3\right),
\end{equation}
and if $d \in \mathbb{Z}$ is coprime to $6$ and satisfies $d \equiv (\frac{n}{(n,24)})^{-1} \bmod \frac{24}{(n,24)}$, then
\begin{equation}
 f_q(n, t_1,t_2,t_3) = f_q\left( (n,24),t_1, t_2 + \binom{d}{2}  t_1^2, t_3 + \binom{d}{3}  t_1^3\right).
\end{equation}
\end{thm}
We prove an analogue of Theorem \ref{24period} for characteristic $5$.
\begin{thm}\label{60period}
Let $q=5^r$. Let $t_1,t_2,t_3 \in \FF_q$. Then 
\begin{equation}\label{periodicthm5}
f_q(n,t_1,t_2,t_3)=\frac{\psi_q(n,t_1,t_2,t_3)- q^{n-3}}{q^{n/2}} \text{ is periodic in }n \in \mathbb{N}_{>0} \text{ with period }60.
\end{equation}
Additionally, the following symmetries hold. Let $n \ge 1$. If $\lambda \in \mathbb{N}_{>0}$ is coprime to $30$, then
\begin{equation}\label{sym15}
f_q(n,t_1,t_2,t_3) = f_q\left(\lambda  n, \lambda t_1,  \lambda  t_2 + \binom{\lambda}{2} t_1^2, \lambda  t_3 + \lambda(\lambda -1) t_1 t_2+ \binom{\lambda}{3} t_1^3\right),
\end{equation}
and if $d \in \mathbb{Z}$ is coprime to $30$ and satisfies $d \equiv (\frac{n}{(n,60)})^{-1} \bmod \frac{60}{(n,60)}$, then
\begin{equation}\label{sym25}
f_q(n, t_1,t_2,t_3) = f_q\left( (n,60),d t_1, d t_2 + \binom{d}{2}  t_1^2, dt_3 + d(d-1) t_1 t_2+ \binom{d}{3}  t_1^3\right).
\end{equation}
\end{thm}
We prove a partial converse for Theorem \ref{24period} and Theorem \ref{60period}.
\begin{thm}\label{conv2460}
Let $\ell$ be a positive integer. Let $q=p^r$ be a prime power. Define the following function in variable $n\in \mathbb{N}_{>0}$:
\begin{equation}\label{fql}
f_{q,\ell}(n) =  \frac{\psi_q(n,\underbrace{0,\ldots, 0}_{\ell})-q^{n-\ell}}{q^{\frac{n}{2}}}.
\end{equation}
If $\ell =3$ and $p \notin \{2,5\}$ or $\ell \ge 4$, then the function $f_{q,\ell}(n)$ is not periodic in $n$.
\end{thm}
From Theorems \ref{24period}--\ref{conv2460} and Proposition \ref{sym2}, we immediately obtain the following result.
\begin{cor}
Let $\ell$ be a positive integer and $q=p^r$ be a prime power. The function $f_{q,\ell}(n)$ defined in \eqref{fql} is periodic in $n$ if and only if one of the following conditions hold.
\begin{enumerate}
\item $\ell = 1$, in which case the period is $1$.
\item $\ell =2$, in which case the period divides $4p$.
\item $\ell = 3$, $p \in \{2,5\}$, in which case the period divides $12 p$.
\end{enumerate}
\end{cor}
We note that by \eqref{eq:psiest}, it make sense to consider the periodicity of the normalized functions $f_q(n,t_1,t_2,t_3)$ and $f_{q,\ell}(n)$, as they are bounded as functions of $n$. 

We discuss an important consequence of Theorem \ref{24period}, namely, that it allows us to obtain formulas for $\psi_{2^r}(n,t_1,t_2,t_3)$ for most values of $n$. The last part of Theorem \ref{24period} reduces the computation of $\psi_{2^r}(n,t_1,t_2,t_3)$ to the computation of $\psi_{2^r}((n,24),\tilde{t_1},\tilde{t_2},\tilde{t_3})$ for some $\tilde{t_i}$. There are 8 possible values of $(n,24)$: 1, 2, 3, 4, 6, 8, 12 and 24. If $(n,24)=1$, then the computation of $\psi_{2^r}((n,24),\tilde{t_1},\tilde{t_2},\tilde{t_3})$ is easy. Indeed, 
\begin{equation}
\psi_{2^r}(1,\tilde{t_1},\tilde{t_2},\tilde{t_3}) = 1_{\tilde{t_2}=\tilde{t_3}=0},
\end{equation}
since there is no monic linear polynomial with first 3 next-to-leading coefficients equal to $\tilde{t_i}$, unless $\tilde{t_2}=\tilde{t_3}=0$, in which case there is one such polynomial: $T+\tilde{t_1}$. The case $(n,24)=2$ is similar -- $\psi_q(2,\tilde{t_1},\tilde{t_2},\tilde{t_3})$ is $0$ if $\tilde{t_3} \neq 0$, and we obtain 
\begin{equation}
\psi_{2^r}(2,\tilde{t_1},\tilde{t_2},\tilde{t_3}) = \Lambda_{2^r}(T^2+\tilde{t_1}T+\tilde{t_2}) \cdot  1_{\tilde{t_3}=0}.
\end{equation}
The case $(n,24)=3$ is also simple. The only monic cubic polynomial with 3 next-to-leading coefficients equal to $\tilde{t_i}$ is $T^3+\tilde{t_1}T^2+\tilde{t_2}T+\tilde{t_3}$, and so
\begin{equation}
\psi_{2^r}(3,\tilde{t_1},\tilde{t_2},\tilde{t_3}) = \Lambda_{2^r}(T^3+\tilde{t_1}T^2+\tilde{t_2}T+\tilde{t_3}),
\end{equation}
which may be computed using the irreducibility criteria for cubic polynomials in characteristic 2 \cite{williams1975}.

If $(n,24)=4$, the evaluation of $\psi_q(4,\tilde{t_1},\tilde{t_2},\tilde{t_3})$ was done by Kuz'min \cite{kuzmin1994}, and sometimes involves Kloosterman sums (depending on $\tilde{t_i}$). If $(n,24)=24$, then we have the following corollary of Theorem \ref{24period} and its proof, proved in \S\ref{sec:proofformula24}.
\begin{cor}\label{cor:char224}
Let $q=2^r$. Let $t_1,t_2,t_3 \in \FF_q$. Then
\begin{equation}
\psi_q(24,t_1,t_2,t_3) = q^{21} - q^{10} (2q^2 \cdot 1_{t_1=t_2=t_3=0} - q \cdot 1_{t_1=t_2=0} - 1_{t_1=0}).
\end{equation}
\end{cor}
The cases $(n,24)\in \{6,8,12\}$ seem harder and we do not treat them here. In characteristic 5, Theorem \ref{60period} similarly reduces the computation of $\psi_{5^r}(n,t_1,t_2,t_3)$ to the computation of $\psi_{5^r}((n,60),\tilde{t_1},\tilde{t_2},\tilde{t_3})$ for some $\tilde{t_i}$. The cases $(n,60) \in \{1,2,3\}$ are again easy, $(n,60)=4$ was done by Kuz'min and $(n,60)=60$ is treated in the following corollary, proved in \S\ref{sec:proofformula60}.
\begin{cor}\label{cor:char560}
Let $q=5^r$. Let $t_1,t_2,t_3 \in \FF_q$. Then
\begin{equation}
\psi_q(60,t_1,t_2,t_3) = q^{57} - q^{28} (2q^2 \cdot 1_{t_1=t_2=t_3=0} - q \cdot 1_{t_1=t_2=0} - 1_{t_1=0}).
\end{equation}
\end{cor}
\subsection{Methods and interpretation}
As in the works of Carlitz and Kuz'min, we use Dirichlet characters and $L$-Functions. This approach dates back to Dirichlet's work on primes in arithmetic progressions. The Dirichlet characters that we use are defined modulo a power of the prime at infinity in $\FF_q(T)$. Concretely, we consider the group $G(R_{\ell,q})$ of characters whose value on $f\in \mathcal{M}_q$ depends only the first $\ell$ next-to-leading coefficients of $f$, see \S\ref{chars} for a precise definition. Lemma \ref{lemformpsi} expresses $f_q(n,t_1,t_2,t_3)$ as an average over $G(R_{3,q})$. 

The main new ingredients in the proofs of Theorems \ref{24period} and \ref{60period} are Lemmas \ref{lfuncroot24} and \ref{artinunity} respectively, which say the following. The roots of the $L$-function $L(u,\chi)$ of a non-trivial character in $G(R_{3,q})$ are of the form $\frac{\omega}{\sqrt{q}}$, where $\omega$ is a root of unity of order dividing $24$ if $q=2^r$, and dividing $60$ if $q=5^r$.

We now explain how Lemma \ref{lfuncroot24} relates to a result of Katz, and complements it. Fix a prime $p$. For any non-trivial $\chi \in G(R_{3,p^r}) \setminus G(R_{2,p^r})$, we can write $L(u,\chi)$ as 
\begin{equation}
L(u,\chi) = \det(I_2-u\sqrt{p^r}\Theta_{\chi}),
\end{equation}
where $\Theta_{\chi} \in U(2)$ is a unitary matrix, or rather a conjugacy class of such a matrix. Katz \cite{katz2013} considered the distribution of the matrices $\{ \Theta_{\chi} \}_{\chi \in G(R_{3,p^r}) \setminus G(R_{2,p^r})}$ in $PU(2) = U(2) / Z(U(2))$ as $r \to \infty$, in the following sense. He showed that if $p \neq 2,5$, then for any continuous function $f\colon U(2) \to \mathbb{C}$, invariant under conjugation and multiplication by scalars, we have
\begin{equation}\label{eq:equi}
\lim_{r \to \infty} \frac{\sum_{\chi\in G(R_{3,p^r})\setminus G(R_{2,p^r})}f(\Theta_{\chi})}{\# G(R_{3,p^r}) \setminus G(R_{2,p^r})} = \int_{U(2)} f(U) \,d U,
\end{equation}
where $dU$ is the unit mass Haar measure on $U(2)$; such a result is known as an \textit{equidistribution result}. For $p=5$, Katz showed that \eqref{eq:equi} fails. For $p=2$ he remarked \cite[Rem.~5.2]{katz2013} that he suspects that \eqref{eq:equi} fails as well. Lemma \ref{lfuncroot24} confirms this suspicion. Indeed, taking $f(U) = |\mathrm{Tr}(U^{24})|^2$, the left hand side of \eqref{eq:equi} is $4$ since $\Theta_{\chi}^{24}=I_2$ for $\chi \in G(R_{3,2^r})\setminus G(R_{2,2^r})$, while the right hand side is $2$ by \cite[Thm.~2.1]{diaconis2001}.

An interpretation of Lemma \ref{lfuncroot24} in terms of supersingularity of certain curves associated with cyclotomic function fields is given in \S\ref{sec:cyclotomic}.

Throughout the paper, we use the following notation. For any $k \in \mathbb{N}_{>0}$ we define 
\begin{equation}
\omega_k = e^{\frac{2\pi i}{k}}.
\end{equation}
For a prime power $q=p^r$, we denote by $\chi_{q}\colon \FF_q \to \mathbb{C}^{\times}$ the additive character defined by
\begin{equation}
\chi_q(a) = \omega_p^{\mathrm{Tr}_{\FF_q/\FF_p}(a)}
\end{equation}
for all $a \in \FF_q$.
\section{Short interval characters}\label{chars}
In $\FF_q(T)$, the prime at infinity is not different from any other prime of degree $1$. In particular, one may define Dirichlet characters modulo a power of the prime at infinity. We call these characters ``short interval characters" and review their properties below. We follow Hayes \cite{hayes1965}, who studied the Dirichlet characters of $\FF_q(T)$ in an elementary manner. 
\subsection{Equivalence relation}\label{erhayes}
For any $\ell \in \mathbb{N}$ we define an equivalence relation $R_{\ell}$ on $\mathcal{M}_q$ by saying that $A \equiv B \bmod R_{\ell}$ if and only if $A$ and $B$ have the same first $\ell$ next-to-leading coefficients. We sometimes write $R_{\ell,q}$ instead of $R_{\ell}$, to emphasize that we work in $\mathcal{M}_q$.

We say that $A \in \mathcal{M}_q$ is invertible modulo $R_{\ell}$ if there exists $B \in \mathcal{M}_q$ such that $AB \equiv 1 \bmod R_{\ell}$. It can be shown that any element of $\mathcal{M}_q$ is invertible modulo $R_{\ell}$.

Consider the set $\mathcal{M}_q / R_{\ell}$ of equivalence classes of $R_{\ell}$. It can be shown that $\mathcal{M}_q / R_{\ell}$ is an abelian group with respect to polynomial multiplication, having as identity element the equivalence class of the polynomial $1$. The group $\mathcal{M}_q / R_{\ell}$ is a $p$-group of order $q^{\ell}$.

We sometimes represent an element $T^n + a_1 T^{n-1} + \ldots + a_n \in \mathcal{M}_q / R_{\ell}$ by an $\ell$-tuple $(a_1, a_2, \ldots, a_{\ell})$. In this notation, multiplication of elements of $\mathcal{M}_q / R_{\ell}$ takes the form
\begin{equation}
(a_1,\ldots, b_{\ell}) (b_1,\ldots, b_{\ell} ) = (a_1 +b_1, a_2 + a_1 b_1 + b_2, \ldots, a_{\ell} +\sum_{i+j=\ell, 1 \le i,j \le \ell - 1} a_i b_j + b_{\ell}).
\end{equation}
The $j$-th coordinate of $(a_1,\ldots, a_{\ell})^k$ is the coefficient of $T^{nk-j}$ in $(T^n+a_1T^{n-1}+a_2 T^{n-2} + \ldots )^k$:
\begin{equation}\label{powermqr}
(a_1,\ldots, a_{\ell})^ k = \left(\binom{k}{1} a_1,  \binom{k}{1} a_2 + \binom{k}{2} a_1^2 ,\binom{k}{1} a_3 + 2\binom{k}{2} a_1 a_2 + \binom{k}{3} a_1^3 , \ldots, \binom{k}{1} a_{\ell} + \ldots \right).
\end{equation}
\subsection{Characters}
For every character $\chi$ of the finite abelian group $\mathcal{M}_q / R_{\ell}$, we define $\chi^{\dagger}$ with domain $\mathcal{M}_q$ as follows: For a given $A \in \mathcal{M}_q$, if $\mathfrak{c}$ is the equivalence class of $A$, then 
\begin{equation}
\chi^{\dagger}(A)=\chi(\mathfrak{c}).
\end{equation}
The functions $\chi^{\dagger}$ defined in this way are called the characters of the relation $R_{\ell}$, or simply characters modulo $R_{\ell}$. We denote by $G(R_{\ell})$ the set containing them:
\begin{equation}
G(R_{\ell}) = \{\chi^{\dagger} : \chi\in \widehat{\mathcal{M}_q / R_{\ell}}\}.
\end{equation}
We will often abuse notation and write $\chi$ instead of $\chi^{\dagger}$. In particular, we denote by $\chi_0$ the trivial character in $G(R_{\ell})$.

Assume that $\ell >0$. A character modulo $R_{\ell}$ is said to be ``primitive modulo $R_{\ell}$", or just ``primitive", if $\chi \notin G(R_{\ell -1})$.

We have a natural restriction map from $\mathcal{M}_q / R_{\ell}$ to $\mathcal{M}_q / R_{\ell-1}$, which in particular implies that $G(R_{\ell-1})$ may be considered as a subgroup of $G(R_{\ell})$.

If $n \ge \ell$ and $\chi \in G(R_{\ell})$, then a consequence of the orthogonality relations in $\mathcal{M}_q / R_{\ell}$ is
\begin{equation}\label{orthouse}
\frac{1}{q^{n}} \sum_{F \in \mathcal{M}_{n,q}} \chi(F) = \begin{cases} 0 & \text{if }\chi \neq \chi_0, \\ 1 & \text{if }\chi = \chi_0. \end{cases}
\end{equation}
\subsection{\texorpdfstring{$L$}{\textit{L}}-functions}\label{seclfunc}
Let $\chi$ be a character of $G(R_{\ell})$. The $L$-function of $\chi$ is the following power series in $u$:
\begin{equation}
L(u,\chi) = \sum_{f \in \mathcal{M}_q} \chi(f)u^{\deg f},
\end{equation}
which also admits the Euler product
\begin{equation}
L(u,\chi) = \prod_{P \in \mathcal{M}_q} (1-\chi(P)u^{\deg P})^{-1}.
\end{equation}
If $\chi$ is the trivial character $\chi_0$, then
\begin{equation}\label{trivl}
L(u,\chi) = \frac{1}{1-qu}.
\end{equation}
Otherwise, \eqref{orthouse} implies that $L(u,\chi)$ is a polynomial in $u$ of degree at most $\ell-1$.

Weil's proof of the Riemann Hypothesis for Function Fields \cite[Thm.~6,~p.~134]{weil1974} implies the Riemann Hypothesis for the $L$-functions of $\chi \in G(R_{\ell})$, as shown explicitly by Rhin \cite[Thm.~3]{rhin1972} in his thesis (cf. \cite[Thm.~5.6]{effinger1991} and the discussion following it). Hence we know that if we factor $L(u,\chi)$ as
\begin{equation}\label{ldecomproot}
L(u,\chi) = \prod_{i=1}^{\deg L(u,\chi)} (1-\gamma_i(\chi)u),
\end{equation}
then the $\gamma_i(\chi)$-s are $q$-Weil numbers of weight 1, i.e. they are algebraic numbers such that
\begin{equation}\label{absweight1}
\left| \gamma_i(\chi) \right| = \sqrt{q},
\end{equation}
and \eqref{absweight1} is true for the conjugates of $\gamma_i(\chi)$ as well. The functional equation for $L(u,\chi)$ implies that if $\chi$ is primitive modulo $R_{\ell}$ then 
\begin{equation}
\deg L(u,\chi)  = \ell-1.
\end{equation}
We use the notation $\gamma_i(\chi)$ for the inverse roots of $L(u,\chi)$ throughout this paper.

\subsection{Lifts of characters}
Let $\sigma\colon \FF_{q^r}\to \FF_{q^r}$ be the Frobenius automorphism $\sigma(\alpha)=\alpha^{q}$, which generates $\text{Gal}(\FF_{q^r}/ \FF_q)$. We extend $\sigma$ to an automorphism of $\FF_{q^r}[T]$ by $\sigma(\sum_i a_iT^i) = \sum_i \sigma(a_i)T^i$. For every $A \in \FF_{q^r}[T]$, the polynomial
\begin{equation}
N^{(r)}(A)=\prod_{i=1}^{r} \sigma^{i}(A)
\end{equation}
is called the norm of $A$ relative to $\FF_q[T]$. As explained in \cite[Thm.~2.2]{hayes1965} for example, $N^{(r)}(A)$ is a polynomial in $\FF_q[T]$. In addition, the map $A \mapsto N^{(r)}(A)$ is multiplicative. It is also true that if $A \equiv B \bmod R_{\ell,q^r}$ then $N^{(r)}(A) \equiv N^{(r)}(B) \bmod R_{\ell,q}$, i.e. the $\ell$ next-to-leading coefficients of the norm of $A$ depend only on the $\ell$ next-to-leading coefficients of $A$.
If $\chi \in G(R_{\ell,q})$, then for every $A \in \FF_{q^r}[T]$ we define 
\begin{equation}
\chi^{(r)}(A)=\chi(N^{(r)}(A)).
\end{equation} 
For every character $\chi$ of $R_{\ell,q}$, $\chi^{(r)}$ is a character of $R_{\ell,q^r}$. The map $\chi \mapsto \chi^{(r)}$ is a homomorphism from $G(R_{\ell,q})$ to $G(R_{\ell,q^r})$ \cite[Thm.~5.2]{hayes1965}.
By \cite[Thm.~9.2]{hayes1965}, if $
L(u,\chi) = \prod_{i=1}^{\deg L(u,\chi)} (1-\gamma_i(\chi)u)$ then
\begin{equation}\label{liftlform}
L(u,\chi^{(r)}) = \prod_{i=1}^{\deg L(u,\chi)} (1-\gamma_i(\chi)^r u).
\end{equation}
The character $\chi^{(r)}$ is called a lift of $\chi$ to $\FF_{q^r}$.
\subsection{Formula for \texorpdfstring{$\psi_q$}{psi q}}
The following lemma gives a classical exact formula for $\psi_q$ \cite[Thms.~6.1,~9.2]{hayes1965}.
\begin{lem}\label{lemformpsi}
Let $n \in \mathbb{N}_{>0}$ and $\ell \in \mathbb{N}$. Let $t_1,\ldots, t_{\ell} \in \FF_q$. Set
\begin{equation}
A = (t_1, t_2, \ldots, t_{\ell}) \in \mathcal{M}_{q} / R_{\ell}.
\end{equation}
For any $\chi_0 \neq \chi \in G(R_{\ell})$, let $\{ \gamma_i(\chi) \}_{i=1}^{\deg L(u,\chi)}$ be the inverse roots of $L(u,\chi)$. Then
\begin{equation}\label{psiformsumchi}
\psi_q(n,t_1,\ldots,t_{\ell}) =q^{n-\ell} - \frac{1}{q^{\ell}} \sum_{\chi_0 \neq \chi \in G(R_{\ell})} \overline{\chi}(A) \sum_{i=1}^{\deg L(u,\chi)} \gamma_i^n(\chi). 
\end{equation}
\end{lem}
Lemma \ref{lemformpsi} implies that for fixed $\ell$ and $q$, only a finite collection of complex numbers governs the behavior of the function $\psi_q(n,t_1,\ldots,t_{\ell})$, namely the multiset $\{ \gamma_{i}(\chi) \}_{\chi \in G(R_{\ell,q}) \setminus \{ \chi_0\},\, 1\le  i \le \deg L(u,\chi) }$. 
For instance, the formulas for $\psi_2(n,t_1,t_2,t_3)$ given in \cite[Thm.~4]{yucas2004} and \cite[Thm.~3.4]{fitzgerald2003} may be recovered easily by a computation of $10$ complex numbers: two inverse roots for each $\chi \in G(R_3) / G(R_2)$ and one inverse root for each $\chi \in G(R_2) / G(R_1)$.
\begin{remark}
From Lemma \ref{lemformpsi} and the Riemann Hypothesis for $\chi \in G(R_{\ell})$, \eqref{eq:psiest} is obtained.
\end{remark}
\subsection{The groups \texorpdfstring{$\mathcal{M}_q/R_{\ell}$}{Mq R l} and \texorpdfstring{$G(R_{\ell})$}{G(R l)}}
Let $q=p^r$ be a prime power. The groups $\mathcal{M}_q/R_{\ell}$ and $G(R_{\ell})$ are abelian $p$-groups of order $q^{\ell}$. They are isomorphic since $G(R_{\ell})$ is the character group of $\mathcal{M}_q/R_{\ell}$.

The decomposition of $\mathcal{M}_q/R_{\ell}$ as a product of cyclic $p$-groups is known, see for instance Katz \cite[Sec.~2]{katz2013} or Fomenko \cite[p.~276]{fomenko1996}. We only need the following simpler result.
\begin{lem}\label{structlem}
Let $q=p^r$ be a prime power and $\ell \ge 1$. The exponent $e_q$ of $\mathcal{M}_q/R_{\ell} \cong G(R_{\ell})$ is 
\begin{equation}
e_q = p^{\lceil \log_p (\ell+1) \rceil}.
\end{equation}
\end{lem}
\begin{proof} 
The exponent of $\mathcal{M}_q/R_{\ell}$ is necessarily a power of $p$. For $A \in \mathcal{M}_q / R_{\ell}$, the power $A^{p^t}$ has first $p^t-1$ next-to-leading coefficients equal to zero. Hence, a sufficient condition for $A^{p^t} = 1$ for all $A$ is 
\begin{equation}\label{expcond}
p^t -1 \ge \ell.
\end{equation}
Let $t$ be the smallest integer satisfying \eqref{expcond}, that is, $t = \lceil \log_p(\ell+1) \rceil$. Then $(T+1)^{p^{t-1}} = T^{p^{t-1}}+1 \neq 1 \bmod R_{\ell}$, which proves that $p^t$ is the exponent of the group.
\end{proof}
When $p > \ell$, the $L$-functions of the characters in $G(R_{\ell,q})$ have a simple description given in the following lemma. The $L$-functions are essentially $L$-functions of Artin-Schreier curves
\begin{lem}\label{aslem}
Let $q=p^r$ be a prime power. Let $\ell$ be a positive integer such that $p>\ell$. Let $\chi \in G(R_{\ell,q})$. Then there exist unique constants $\lambda_1, \ldots, \lambda_{\ell} \in \FF_q$ such that 
\begin{equation}\label{charstruct}
\chi(a_1,\ldots,a_{\ell}) = \chi_{q}\left( \sum_{j=1}^{\ell} \lambda_j \cdot [T^j] \log(1+a_1 T + a_2 T^2 + \ldots + a_{\ell} T^{\ell}) \right)
\end{equation}
for all $a_i \in \FF_q$. Conversely, any function on $\mathcal{M}_q / R_{\ell}$ defined as in \eqref{charstruct} is an element of $G(R_{\ell,q})$.

Here $[T^j] \log(1+a_1 T + a_2 T^2 + \ldots + a_{\ell} T^{\ell})$ denotes the coefficient of $T^j$ in the formal power series $\log(1+a_1 T + a_2 T^2 + \ldots + a_{\ell}T^{\ell})$. As $j \le \ell < p$, this coefficient is well defined. 
\end{lem}
This lemma is implicit in the works of Katz and Fomenko mentioned above. We supply a proof for completeness.
\begin{proof}
The characters of $\FF_q^{\ell}$ are given by $(b_1,\ldots,b_{\ell}) \mapsto \chi_q(\sum_{j=1}^{\ell} \lambda _j b_j)$ for $\lambda_i \in \FF_q$. The lemma is deduced by observing that $\mathcal{M}_q/R_{\ell}$ is isomorphic to $\{ \sum_{i=1}^{\ell} c_i T^i : c_i \in \FF_q\} \cong \FF_q^{\ell}$ under
\begin{equation}\label{eq:homom}
T^{\ell}+a_1 T^{\ell-1} + \ldots + a_{\ell}  \mapsto \log(1+a_1T+\ldots+a_{\ell}T^{\ell}) \bmod T^{\ell+1},
\end{equation}
where $\log(1+a_1T+\ldots+a_{\ell}T^{\ell}) \bmod T^{\ell+1}$ is a truncation of $\log$, well defined since $j \le \ell <p$. The fact that \eqref{eq:homom} is an homomorphism follows from the fact that log turns multiplication into addition. The fact that it is an isomorphism follows from the fact that we have an explicit inverse, given by (a truncation of) the exponential map.
\end{proof}

\section{Periodicity and symmetry for \texorpdfstring{$p=2$}{p=2}, \texorpdfstring{$\ell=3$}{l=3}}
\subsection{Symmetries of sums involving roots of unity}
The following lemma describes some symmetries of expressions involving roots of unity.
\begin{lem}\label{symrootlior}
Let $m_1$, $m_2$ be positive integers with $m_2 \mid m_1$. Let $F_1,\ldots, F_{m_1} \in \mathbb{Q}[x]$. Let $S\colon \mathbb{Z}^2 \to \mathbb{C}$ be the function defined by
\begin{equation}
S(k_1,k_2) = \sum_{i=1}^{m_1} \omega_{m_1}^{i k_1} F_i(\omega_{m_2}^{k_2}).
\end{equation}
Assume that, for all $k_1, k_2 \in \mathbb{Z}$,
\begin{equation}\label{ratsum}
S(k_1,k_2) \in \mathbb{Q}.
\end{equation}
Then the following hold.
\begin{enumerate}
\item The function $S(k_1,k_2)$ is periodic in $k_1$ with period $m_1$ and in $k_2$ with period $m_2$.
\item For any $\lambda \in \mathbb{Z}$ such that $(\lambda,m_1)=1$, we have
\begin{equation}
S(k_1, k_2)=S(\lambda  k_1, \lambda  k_2).
\end{equation}
\item Let $k_1, k_2 \in \mathbb{Z}$. Let $d$ be any integer such that $d \equiv (\frac{k_1}{(k_1,m_1)})^{-1} \bmod \frac{m_1}{(k_1,m_1)}$ and $(d,m_1)=1$. Then
\begin{equation}
S(k_1,k_2) = S( (k_1,m_1), d  k_2).
\end{equation}
\end{enumerate}
\end{lem}
\begin{proof}
The first part is trivial. For the second part, let $\sigma_{\lambda}\colon \mathbb{Q}(\omega_{m_1}) \to \mathbb{Q}(\omega_{m_1})$ be the field automorphism fixing $\mathbb{Q}$ defined by
\begin{equation}
\sigma_{\lambda}(\omega_{m_1}) = \omega_{m_1}^{\lambda}.
\end{equation}
We have 
\begin{equation}\label{galact}
\sigma_{\lambda}(S(k_1, k_2))=S(\lambda  k_1, \lambda  k_2). 
\end{equation}
From \eqref{ratsum} and \eqref{galact}, we deduce the second part of the lemma. To establish the third part, choose in the second part $\lambda=d$ and use the periodicity of $S(k_1,k_2)$ in $k_1$. 
\end{proof}

\subsection{Periodicity and symmetry for \texorpdfstring{$\ell =1,2$}{ell =1,2}}\label{periodsmallell}
We first consider $\psi_q(n,t_1,\ldots,t_{\ell})$ for $\ell=1$. Let $q=p^r$ be a prime power and let $\chi \in G(R_{1})$ be a non-trivial character. Then $\deg L(u,\chi) \le 1-1$. Hence
\begin{equation}
L(u,\chi) = 1.
\end{equation}
In particular, $L(u,\chi)$ has no roots. Applying Lemma \ref{lemformpsi} with $\ell =1$, we obtain that
\begin{equation}\label{eq:NewCarlitz}
\psi_q(n,t_1) = q^{n-1},
\end{equation}
a result first proved by Carlitz \cite[Eq.~6.1]{carlitz1952}. In particular, $\frac{\psi_q(n,t_1)-q^{n-1}}{q^{n/2}}$ has period $1$ in $n$ for any $t_1 \in \FF_q$.

In Proposition \ref{sym2} we prove some periodicity and symmetry properties of $\psi_q(n,t_1,t_2)$. Although these properties follow from the exact formulas for $\psi_q(n,t_1,t_2)$ obtained by Kuz'min, we take a different route. We do this because similar ideas are used in the proof of Theorem \ref{24period}, which is a natural sequel to Proposition \ref{sym2}.
\begin{proposition}\label{sym2}
Let $q=p^r$ be a prime power. Let $t_1, t_2 \in \FF_q$ and define
\begin{equation}
f_q(n,t_1,t_2) = \frac{\psi_q(n,t_1,t_2)-q^{n-2}}{q^{n/2}}.
\end{equation}
Then the following hold.
\begin{enumerate}
\item The function $f_q$ may be written as
\begin{equation}
f_q(n,t_1,t_2) = \frac{-1}{q^2}\sum_{\chi \in G(R_2) \setminus G(R_1)} \overline{\chi}(t_1,t_2) \left( \frac{\gamma_1(\chi)}{\sqrt{q}} \right)^n,
\end{equation}
where
\begin{equation}
\gamma_1(\chi) = -\sum_{a \in \FF_q}\chi(T+a).
\end{equation}
\item Let 
\begin{equation}
o_q =\begin{cases} 4p & \mbox{if $2 \nmid r$ and $p \not\equiv 1 \bmod 4$}, \\ 2p & \mbox{otherwise.} \end{cases} 
\end{equation}
Then $\frac{\gamma_1(\chi)}{\sqrt{q}}$ is a root of unity of order dividing $o_q$. Moreover, if $o_q=4p$ then the order of $\frac{\gamma_1(\chi)}{\sqrt{q}}$ is divisible by $(o_q,8)$.
\item The function $f_q(n,t_1,t_2)$ has period $o_q$ in $n\in \mathbb{N}_{>0}$. Additionally, the following symmetries hold. Let $n \in \mathbb{N}_{>0}$. If $\lambda \in \mathbb{N}_{>0}$ satisfies $(\lambda,2p)=1$, then
\begin{equation}
f_q(n,t_1,t_2) = f_q\left(\lambda  n, \lambda t_1, \lambda t_2 + \binom{\lambda}{2} t_1^2\right),
\end{equation}
and if $d\in \mathbb{Z}$ satisfies $(d,2p)=1$ and $d \equiv (\frac{n}{(n,o_q)})^{-1} \bmod \frac{o_q}{(n,o_q)}$, then
\begin{equation}
f_q(n, t_1,t_2) = f_q\left( (n,o_q), d t_1, dt_2 + \binom{d}{2}t_1^2\right).
\end{equation}
\end{enumerate}
\end{proposition}
The proof we give for the second part of Proposition \ref{sym2} is essentially due to Katz \cite[p.~3631]{katz2013}, although Katz was only interested in odd $p$.
\begin{proof}
The first part of the proposition follows directly from Lemma \ref{lemformpsi} with $\ell=2$. We proceed to prove the second part. 
We start by showing that $\gamma_1(\chi)$ has absolute value $\sqrt{q}$ whenever $\chi \in G(R_2) \setminus G(R_1)$. We have
\begin{equation}\label{gamma1eval}
\begin{split}
|\gamma_1(\chi)|^2 &= \gamma_1(\chi) \overline{\gamma_1(\chi)} = \sum_{a,b \in \FF_q} \chi(T+a) \overline{\chi}(T+b) = \sum_{a,b \in \FF_q} \chi(T+a) \chi(T^2-bT+b^2) \\
&= \sum_{a,b \in \FF_q} \chi(-b+a, b(b-a) ) = \sum_{a' ( =a-b),b' (= -b)\in \FF_q} \chi(a',a'b').
\end{split} 
\end{equation}
We observe that the terms in the last sum of  \eqref{gamma1eval} corresponding to some fixed $a' \neq 0$ contribute $\sum_{c \in \FF_q} \chi(a',c)$ to the sum. Hence, \eqref{gamma1eval} may be written as follows.
\begin{equation}\label{gamma1eval2}
|\gamma_1(\chi)|^2= \sum_{0 \neq a' \in \FF_q, c \in \FF_q} \chi(a',c) + \sum_{b' \in \FF_q} \chi(0,0) = \sum_{x,y \in \FF_q} \chi(x,y) - \sum_{c \in \FF_q} \chi(0,c) + q.
\end{equation}
We have
\begin{equation}\label{van1ri}
\sum_{x,y \in \FF_q} \chi(x,y) = 0
\end{equation}
by the orthogonality relation \eqref{orthouse}. The identity $\chi(a,b) = \chi(a,0) \cdot \chi(0,b)$ implies that  $\chi$ cannot be constantly $1$ on the subgroup $\{ (0,c) : c \in \FF_q\} \subseteq \mathcal{M}_q / R_2$, since otherwise $\chi \in G(R_1)$. Hence, orthogonality of characters implies that 
\begin{equation}\label{van2ri}
\sum_{c \in \FF_q} \chi(0,c)=0.
\end{equation}
From \eqref{gamma1eval2}, \eqref{van1ri} and \eqref{van2ri} we get that
\begin{equation}\label{gamma1absval}
\gamma_1(\chi) \overline{\gamma_1(\chi)} = q.
\end{equation}
Let
\begin{equation}
e_q=\begin{cases} p & \text{if }p > 2, \\  2p & \text{if }p=2, \end{cases}
\end{equation}
be the exponent of the group $G(R_{2,q})$, according to Lemma \ref{structlem}. Since $\chi(a,b)^{e_q} = \chi( (a,b)^{e_q}) = \chi(0,0)=1$, it follows that $\chi(a,b) \in \mathbb{Z}[\omega_{e_q}]$ and so
\begin{equation}\label{gamma1ring}
\gamma_1(\chi), \overline{\gamma_1(\chi)} \in \mathbb{Z}[\omega_{e_q}].
\end{equation}
Since the image of $\chi$ is contained in $\mathbb{Z}[\omega_{e_q}]$, $\sigma \circ \chi$ is a character of $G(R_2) \setminus G(R_1)$ for any $\sigma \in \text{Gal}(\QQ(\omega_{e_q}) / \QQ)$. Repeating the arguments that gave us \eqref{gamma1absval} we obtain
\begin{equation}\label{gamma1absvalall}
|\gamma_1(\sigma \circ \chi)| = |\sigma(\gamma_1(\chi))| = \sqrt{q}
\end{equation}
for all $\sigma \in \text{Gal}(\QQ(\omega_{e_q}) / \QQ)$. 
Let $\mathfrak{p}$ be the unique prime in $\mathbb{Z}[\omega_{e_q}]$ lying over $p$. From \eqref{gamma1absval} and \eqref{gamma1ring}, $\gamma_1(\chi)$ is an algebraic integer with $|\gamma_1(\chi)|_{\mathfrak{q}} = 1$ for any finite prime $\mathfrak{q}$ in $\mathbb{Q}(\omega_{e_q})$, except for $\mathfrak{p}$. Set
\begin{equation}
g_{p} = \begin{cases} \sum_{a \in \FF_p} \left( \frac{a}{p} \right) \omega_p^a & \text{if }p >2 \\ 1+i & \text{if }p =2 \end{cases} \in \ZZ[\omega_{e_q}].
\end{equation} 
Since $g_p$ is a quadratic Gauss sum when $p>2$, we have \cite[Ch.~6,~Thm.~1]{ireland1990}
\begin{equation}\label{gaussvalue}
g_p = \sqrt{p} \cdot \begin{cases} 1 & \text{if } p \equiv 1 \bmod 4, \\ i=\omega_4 & \text{if }p \equiv 3 \bmod 4, \\ \frac{1+i}{\sqrt{2}}=\omega_8 & \text{if }p=2. \end{cases}
\end{equation}
In particular, 
\begin{equation}\label{eq:gp_abs}
|g_p| = \sqrt{p}.
\end{equation}
Let
\begin{equation}
c_{\chi} = \frac{\gamma_1(\chi)}{g_p^r} \in \QQ(\omega_{e_q}).
\end{equation}
Since $g_p$ is an algebraic integer in $\mathfrak{p}$, it follows that $|c_{\chi}|_{\mathfrak{q}} = 1$ for all finite primes $\mathfrak{q}$ in $\mathbb{Q}(\omega_{e_q})$, except possibly at $\mathfrak{p}$. As $|\sigma(\gamma_1(\chi))|=|g_p^r|=\sqrt{q}$ for all $\sigma \in \mathrm{Gal}(\mathbb{Q}(\omega_{e_q})/\mathbb{Q})$ by \eqref{gamma1absvalall} and \eqref{eq:gp_abs}, it follows that $|c_{\chi}|_{\mathfrak{q}}=1$ at the infinite primes also. Thus, the product formula $1 = \prod_{\mathfrak{q} \text{ a prime of } \QQ(\omega_{e_q})} |c_{\chi}|_{\mathfrak{q}}$ implies that $|c_{\chi}|_{\mathfrak{p}}=1$. Hence $c_{\chi}$ is an algebraic integer, in fact a unit in $\mathbb{Z}[\omega_{e_q}]$.
 
A result of Kronecker \cite{kronecker1857} states that if all the conjugates of an algebraic integer have absolute value 1, then it is a root of unity. Applying this result to $c_{\chi}$ shows that $c_{\chi}$ is a root of unity. The roots of unity in $\mathbb{Q}(\omega_{e_q})$ are powers of $\omega_{e_q\cdot(1+1_{p>2})}$. Hence, for any $\chi \in G(R_2) \setminus G(R_1)$,
\begin{equation}\label{valuecchi}
\exists j \in \mathbb{N}: \frac{\gamma_1(\chi)}{\sqrt{q}} = \omega_{e_q\cdot (1+1_{p>2})}^j \cdot \left( \frac{g_p}{\sqrt{p}} \right)^r.
\end{equation}
Combining \eqref{gaussvalue} and \eqref{valuecchi}, the second part of the proposition follows.

We turn to the last part of the proposition.  For any $A=(t_1,t_2) \in R_{2,q}$ and $\chi \in G(R_{2})/G(R_1)$, the number $\overline{\chi}(A)$ is a root of unity of order (dividing) $e_q$, hence we may write
\begin{equation}
\overline{\chi}(A) = \omega_{e_q}^{a_{\chi,A}},
\end{equation}
where $a_{\chi,A}$ is an integer satisfying $0 \le a_{\chi,A} < e_q$.  For any $0 \le i < o_q$, let
\begin{equation}
F_i(x) = \frac{-1}{q^2}\sum_{\chi \in G(R_{2}) \setminus G(R_{1}): \frac{\gamma_1(\chi)}{\sqrt{q}}= \omega_{o_q}^i} x^{a_{\chi,A}}.
\end{equation}
Set $m_1= o_q$, $m_2=e_q$ and define the following function from $\mathbb{Z}^2$ to $\mathbb{C}$:
\begin{equation}
S(k_1,k_2) = \sum_{i=1}^{m_1} \omega_{m_1}^{ik_1} F_i(\omega_{m_2}^{k_2}).
\end{equation}
We have
\begin{equation}
\begin{split}
S(k_1,k_2) &= \frac{-1}{q^2} \sum_{i=1}^{o_q} \sum_{\chi \in G(R_2)\setminus G(R_1): \frac{\gamma_1(\chi)}{\sqrt{q}} = \omega_{o_q}^i } \omega_{o_q}^{ik_1} \omega_{e_q}^{k_2 \cdot a_{\chi,A}} \\
&= \frac{-1}{q^2} \sum_{\chi \in G(R_2)\setminus G(R_1)} (\frac{\gamma_1(\chi)}{\sqrt{q}})^{k_1} \overline{\chi}(A)^{k_2}\\
& = f_q\left(k_1, k_2 t_1, k_2 t_2 + \binom{k_2}{2} t_1^2\right).
\end{split}
\end{equation}
The third part of the lemma now follows by applying Lemma \ref{symrootlior} to $S$.
\end{proof}
\begin{remark}
Let $q=2^r$. Proposition \ref{sym2} may be used in order to calculate $\psi_q(n,t_1,t_2)$ explicitly, which recovers \cite[Thms.~2--4]{kuzmin1991}. Namely, the third part of the proposition shows that the calculation of $\psi_q(n,t_1,t_2)$ reduces to the following cases: $n=1,2,4$, and if $2\nmid r$ then also the case $n=8$. The cases $n=1,2$ are trivial. By Proposition \ref{sym2} we can treat $n=8$ as follows.
\begin{equation}
\begin{split}
\frac{\psi_q(8,t_1,t_2) - q^6}{q^4} &= \frac{-1}{q^2} \sum_{\chi \in G(R_2) \setminus G(R_1)} \overline{\chi}(t_1,t_2) \left( \frac{\gamma_1(\chi)}{\sqrt{q}} \right)^8 = \frac{-1}{q^2} \sum_{\chi \in G(R_2) \setminus G(R_1)} \overline{\chi}(t_1,t_2) \\
&= \frac{-1}{q^2} \sum_{\chi \in G(R_2)} \overline{\chi}(t_1,t_2) + \frac{1}{q^2} \sum_{\chi \in G(R_1)} \overline{\chi}(t_1,t_2)\\
&=-1_{(t_1,t_2)=(0,0)} + \frac{1_{t_1=0}}{q}.
\end{split}
\end{equation}
So the only non-trivial case is $n=4$. An appropriate linear change of variables $f(T) \mapsto \frac{f(\lambda_1 T+\lambda_2)}{\lambda_2^4}$ ($\lambda_1 \in \FF_q^{\times}, \lambda_2 \in \FF_q$) shows that
\begin{equation}\label{eq:afterlinear}
\psi_q(4,a,b) = \begin{cases} \psi_q(4,1,0) & \mbox{if $a \neq 0$,} \\ \psi_q(4,0,1) & \mbox{if $a=0$, $b \neq $0,} \\ \psi_q(4,0,0) & \mbox{if $a=b=0$,}\end{cases}
\end{equation}
and reduces the problem to the evaluation of $\psi_q(4,1,0)$, $\psi_q(4,0,1)$ and $\psi_q(4,0,0)$. The values $\psi_q(4,0,0)$ and $\psi_q(4,0,1)$ are calculated in  \cite[Thm.~1,~Prop.~2--3]{kuzmin1989} by elementary considerations. The relation
\begin{equation}
q^4 = \psi_q(4) =\sum_{a,b \in \FF_q} \psi_q(4,a,b)= \psi_q(4,0,0) + (q-1)\psi_q(4,0,1) + (q^2-q) \psi_q(4,1,0),
\end{equation}
implied by \eqref{eq:afterlinear}, allows us to recover $\psi_q(4,1,0)$ as well.
\end{remark}
\subsection{Properties of \texorpdfstring{$G(R_{3,2^r})$}{G(R 3, 2r)}}
Fomenko has shown, by elementary means, that the $L$-functions of $\chi \in G(R_{3,2^r})$ satisfy the Riemann Hypothesis \cite[Prop.~5]{fomenko1996}. He did so by using a certain homomorphism 
\begin{equation}
\varepsilon \colon G(R_{3,2^r}) \to \FF_{2^r} \oplus \FF_{2^r}.
\end{equation}
We recall the definition and properties of $\varepsilon$, as proven in \cite[Prop.~5]{fomenko1996}, because we build upon them in the proof of Theorem \ref{24period}. Let $q = 2^r$. Let $a,b,c,d \in \FF_q$. In $\mathcal{M}_q / R_3$ we have
\begin{equation}
(0,a,b) \cdot (0,c,d) = (0,a+b,c+d),
\end{equation}
hence $H(R_{3,q})=\{(0,a,b) :  a,b\in \mathbb{F}_q\}$ is a subgroup of $\mathcal{M}_q / R_3$, isomorphic to $\FF_q \oplus \FF_q$. The restriction $\chi_H$ of any character $\chi \in G(R_{3,q})$ to $H$ is a character of $H$, given by 
\begin{equation}
\chi_H(0,a,b) = \chi_q(\lambda a +\mu b)
\end{equation}
for some unique $\lambda, \mu \in \FF_q$. Fomenko defined the homomorphism $\varepsilon \colon G(R_{3,q}) \to \FF_q \oplus \FF_q$ by 
\begin{equation}
\varepsilon(\chi) = (\lambda , \mu).
\end{equation}
\begin{proposition}[Fomenko]\label{fomenkprop}
Let $q=2^r$. Let $\chi \in G(R_{3,q})$ and assume that $\varepsilon(\chi) = (\lambda,\mu)$.
\begin{enumerate}
\item The homomorphism $\varepsilon$ is surjective and its kernel is of size $q$.
\item If $\lambda=\mu = 0$ then $\chi \in G(R_1)$.
If $\mu =0$ but $\lambda \neq 0$ then $\chi \in G(R_2) \setminus G(R_1)$.
\item If $\mu \neq 0$ then $\chi \in G(R_3)\setminus G(R_2)$. If we write the $L$-function of $\chi$ as
\begin{equation}
L(\frac{u}{\sqrt{q}},\chi) = 1 + \alpha(\chi) u + \beta(\chi)u^2,
\end{equation}
then
\begin{equation}\label{alphasquaredprop}
\alpha(\chi)^2 = \chi\left(\frac{\lambda}{\mu},0,0\right) + \sum_{c \in \FF_q: \mu c^3 + \lambda c^2 + 1 = 0} \chi(c,0,0)
\end{equation}
and
\begin{equation}
\beta(\chi) = \chi(\frac{\lambda}{\mu},0,0).
\end{equation}
\item If $\chi$ is primitive modulo $R_3$ then $L(u,\chi)$ satisfies the Riemann Hypothesis.
\end{enumerate}
\end{proposition}

The proof of Theorem \ref{24period} relies crucially on the following lemma.
\begin{lem}\label{lfuncroot24}
Let $q=2^r$. Let $\chi$ be a primitive character modulo $R_{3,q}$. If we factor $L(u,\chi)$ as 
\begin{equation}
(1-\gamma_1(\chi)u)(1-\gamma_2(\chi)u).
\end{equation}
then 
\begin{equation}
\frac{\gamma_i(\chi)}{\sqrt{q}}
\end{equation}
are roots of unity of order (dividing) $24$.
\end{lem}
\begin{proof}
Let $\lambda, \mu \in \FF_q$ be the elements such that $\varepsilon(\chi) = (\lambda,\mu)$. We know that $\mu \neq 0$, as $\chi$ is primitive. The polynomial 
\begin{equation}
F(x)=\mu x^3 + \lambda x^2 +1
\end{equation}
is separable, since $(F,F')=(F,\mu x^2) = 1$. Let $\sqrt{\chi(\frac{\lambda}{\mu},0,0)}$ be some quadratic root of $\chi(\frac{\lambda}{\mu},0,0)$. By Proposition \ref{fomenkprop} we have
\begin{equation}
G(u):=L\left(\frac{u}{\sqrt{q \cdot \chi(\frac{\lambda}{\mu},0,0)}},\chi\right) = 1+ \sqrt{S_{\chi}} u+u^2,
\end{equation}
where $\sqrt{S_{\chi}}$ is one of the quadratic roots of
\begin{equation}
S_{\chi} = 1 + \frac{\sum_{c \in \FF_q: \mu c^3 + \lambda c^2 + 1} \chi(c,0,0)}{\chi(\frac{\lambda}{\mu},0,0)}.
\end{equation}
By the Riemann Hypothesis for $L(u,\chi)$, we must have that $S_{\chi}$ is real and
\begin{equation}
S_{\chi} \ge 0,
\end{equation}
since $\sqrt{S_{\chi}}$ is a sum of a number on the unit circle and its complex conjugate. 

The group $\mathcal{M}_q / R_{3,q}$ has exponent $4$. Hence $\chi(\frac{\lambda}{\mu},0,0)$ is a root of unity of order (dividing) $4$, and so its square roots are roots of unity of order (dividing) $8$, which divides $24$. So we can reduce the lemma to showing that the roots of $G(u)$ are roots of unity of order dividing $24$.
\begin{enumerate}
\item If $F(x)=0$ has no roots in $\FF_q$ then $S_{\chi}=1$ and $G(u) = 1 \pm u + u^2 \mid u^6-1$. In particular, the roots of $G$ are roots of unity of order dividing $24$.
\item If $F(x)=0$ has exactly one solution in $\FF_q$ then $S_{\chi} \in \{ 1+a : a \in \mu_4\}$. The two complex options $1\pm i$ are impossible, hence we have $G(u) = 1 +u^2$ or $G(u) = 1 \pm \sqrt{2} u +u^2 \mid u^4+1 \mid u ^8-1$. In either case, the roots of $G$ are roots of unity of order dividing $24$.
\item If $F(x)=0$ has exactly two solutions in $\FF_q$, namely $c_1$ and $c_2$, we reach a contradiction since Vieta's formulas and separability of $F$ guarantee that $\frac{\lambda}{\mu}+c_1+c_2$ is a third distinct solution to $F(x)=0$.
\item If $F(x)=0$ has three distinct solutions in $\FF_q$, then $S_{\chi} \in \{ 1+ a+b+c : a,b,c \in \mu_4 \}$. We know that $S_{\chi}$ is real and $\ge 0$, which leaves us only the options
\begin{equation}
S_{\chi} \in \{0,2,4\}.
\end{equation}
Hence $G(u) = 1+u^2$ or $G(u) = 1\pm \sqrt{2} u+u^2 \mid u^4+1$ or $G(u) = 1\pm 2u+u^2 =(u\pm 1)^2$. In all cases the roots of $G$ are roots of unity of order dividing $24$.
\end{enumerate}
Regardless of the number of solutions to $F(x)=0$ in $\FF_q$, the lemma is proven.
\end{proof}
\subsection{Proof of Theorem \ref{24period}}\label{proof24period}
Let $q=2^r$. From Lemma \ref{lemformpsi} we may write
\begin{equation}\label{fqform}
f_q(n,t_1,t_2,t_{3})= \frac{-1}{q^{3}} \sum_{\chi_0 \neq \chi \in G(R_{3})} \overline{\chi}(t_1,t_2,t_3) \sum_{j=1}^{\deg L(u,\chi)} \left(\frac{\gamma_j(\chi)}{\sqrt{q}} \right)^n.
\end{equation}
Proposition \ref{sym2} and Lemma \ref{lfuncroot24} tell us that $\frac{\gamma_j(\chi)}{\sqrt{q}}$ are roots of unity of order dividing $24$. Fix $A=(t_1,t_2,t_3) \in \mathcal{M}_q / R_{3,q}$. By Lemma \ref{structlem}, the group $\mathcal{M}_q / R_{3,q}$ has exponent $4$. Hence, for any $\chi_0 \neq \chi \in G(R_{3})$, the number $\overline{\chi}(A)$ is a root of unity of order dividing $4$. Thus we may write
\begin{equation}
\overline{\chi}(A) = \omega_{4}^{a_{\chi,A}},
\end{equation}
where $a_{\chi,A}$ is an integer satisfying $0 \le a_{\chi,A} < 4$.  For any $0 \le i < 24$, let
\begin{equation}
F_i(x)= \frac{-1}{q^{3}} \sum_{\chi_0 \neq \chi \in G(R_{3})} \sum_{\substack{1 \le j \le \deg L(u,\chi):\\ \frac{\gamma_j(\chi)}{\sqrt{q}}= \omega_{24}^i}} x^{a_{\chi,A}}.
\end{equation}
Set $m_1= 24$, $m_2=4$ and define the following function from $\mathbb{Z}^2$ to $\mathbb{C}$:
\begin{equation}
S(k_1,k_2) = \sum_{i=1}^{m_1} \omega_{m_1}^{ik_1} F_i(\omega_{m_2}^{k_2}).
\end{equation}
As $A^{k_2}=(k_2 t_1, k_2 t_2 + \binom{k_2}{2} t_1^2, k_2 t_3 +\binom{k_2}{3} t_1)$, we have by \eqref{fqform}
\begin{equation}
\begin{split}
S(k_1,k_2) &= \frac{-1}{q^3} \sum_{i=1}^{24} \sum_{\chi_0 \neq \chi \in G(R_3)} \sum_{1 \le j \le \deg L(u,\chi):  \frac{\gamma_j(\chi)}{\sqrt{q}} = \omega_{24}^i } \omega_{24}^{ik_1} \omega_{4}^{k_2 \cdot a_{\chi,A}} \\
&= \frac{-1}{q^3} \sum_{\chi_0 \neq \chi \in G(R_3)}  \overline{\chi}(A^{k_2}) \sum_{1 \le j \le \deg L(u,\chi)} (\frac{\gamma_j(\chi)}{\sqrt{q}})^{k_1} \\
& = f_q\left(k_1, k_2 t_1, k_2 t_2 + \binom{k_2}{2} t_1^2, k_2 t_3 +\binom{k_2}{3} t_1\right).
\end{split}
\end{equation}
The theorem now follows by applying Lemma \ref{symrootlior} to $S$. \qed
\subsection{Proof of Corollary \ref{cor:char224}}\label{sec:proofformula24}
Let $q=2^r$. Let $t_1,t_2,t_3 \in \FF_q$. For any $\chi_0 \neq \chi \in G(R_3)$ and any inverse root $\gamma$ of $L(u,\chi)$, Proposition \ref{sym2} and Lemma \ref{lfuncroot24} tell us that $\frac{\gamma}{\sqrt{q}}$ is a  root of unity of order dividing $24$. Thus, Lemma \ref{lemformpsi} implies that
\begin{equation}\label{24caseeq}
\begin{split}
f_q(24,t_1,t_2,t_3)&= \frac{-1}{q^{3}} \sum_{\chi_0 \neq \chi \in G(R_{3})} \overline{\chi}(t_1,t_2,t_3) \sum_{i=1}^{\deg L(u,\chi)} \left( \frac{\gamma_i(\chi)}{\sqrt{q}} \right)^{24} \\
&= \frac{-1}{q^{3}} \sum_{\chi_0 \neq \chi \in G(R_{3})} \overline{\chi}(t_1,t_2,t_3) \sum_{i=1}^{\deg L(u,\chi)} 1 \\
&= \frac{-1}{q^{3}} \left( \sum_{\chi \in G(R_{2}) \setminus G(R_1)} \overline{\chi}(t_1,t_2,t_3) + 2\sum_{\chi \in G(R_{3})\setminus G(R_2)} \overline{\chi}(t_1,t_2,t_3) \right).
\end{split}
\end{equation}
By \eqref{orthouse} we have
\begin{equation}\label{orthqqqqqq}
\begin{split}
\sum_{\chi \in G(R_1)} \overline{\chi}(t_1,t_2,t_3) &= q \cdot 1_{t_1 = 0},\\
\sum_{\chi \in G(R_2)} \overline{\chi}(t_1,t_2,t_3) &= q^2 \cdot 1_{t_1 =t_2= 0},\\
\sum_{\chi \in G(R_3)} \overline{\chi}(t_1,t_2,t_3) &= q^3 \cdot 1_{t_1 = t_2=t_3=0}.
\end{split}
\end{equation}
Substituting \eqref{orthqqqqqq} in \eqref{24caseeq}, we conclude the proof of the corollary. \qed
\section{Periodicity and symmetry for \texorpdfstring{$p=5$}{p=5}, \texorpdfstring{$\ell=3$}{l=3}}
\subsection{Auxiliary lemmas} 
\begin{lem}\label{lstruct3}
Let $q=p^r$ be a prime power with $p \ge 5$. Let $\chi \in G(R_{3,q}) \setminus G(R_{2,q})$. 
Then there exist $a \in \FF_q^{\times}$, $b,c \in \FF_q$ such that
\begin{equation}
L\left(\frac{u}{\sqrt{q}\omega_{p}^c},\chi\right) = 1 + \frac{\sum_{x \in \FF_q} \chi_q(ax^3 +bx)}{\sqrt{q}} + u^2.
\end{equation}
Moreover, if $\lambda_1, \lambda_2,\lambda_3 \in \FF_q$ are the constants associated with $\chi$ in Lemma \ref{aslem}, then we may take
\begin{equation}
a=\frac{\lambda_3}{3}, \, b=\lambda_1 - \frac{\lambda_2^2}{4\lambda_3}.
\end{equation}
\end{lem}
\begin{proof}
By applying Lemma \ref{aslem} with $\ell=3$ we get that there are $\lambda_1, \lambda_2, \lambda_3 \in \FF_q$ such that
\begin{equation}
L(u,\chi) = 1+ g_1(\lambda_1,\lambda_2,\lambda_3) u + g_2(\lambda_1,\lambda_2,\lambda_3)u^2,
\end{equation}
where
\begin{equation}
\begin{split}
g_1(\lambda_1,\lambda_2,\lambda_3)&=\sum_{x \in \FF_q} \chi_q(\lambda_1 x - \frac{1}{2}\lambda_2 x^2 +\frac{1}{3}\lambda_3 x^3), \\
g_2(\lambda_1,\lambda_2,\lambda_3)&= \sum_{x,y \in \FF_q} \chi_q(\lambda_1 x + \lambda_2 (y -\frac{1}{2}x^2)+\lambda_3 (\frac{1}{3}x^3 - xy)),
\end{split}
\end{equation}
and $\lambda_3 \neq 0$ (otherwise $\chi \in G(R_2)$). As observed by Fomenko \cite[p.~269]{fomenko1996}, we may simplify $g_2$ as follows:
\begin{equation}
\begin{split}
g_2(\lambda_1,\lambda_2,\lambda_3)& = \sum_{x \in \FF_q}\chi_q(\lambda_1 x -\frac{1}{2}\lambda_2 x^2+\frac{1}{3}\lambda_3 x^3) \left( \sum_{y \in \FF_q} \chi_q\left( y\left(\lambda_2- \lambda_3 x\right)\right) \right)\\
&=q \cdot \chi_q\left( \lambda_1 \frac{\lambda_2}{\lambda_3} - \frac{1}{2}\lambda_2 \left( \frac{\lambda_2}{\lambda_3}\right)^2 + \frac{1}{3}\lambda_3 \left( \frac{\lambda_2}{\lambda_3} \right)^3 \right)\\
&=q \cdot \chi_q \left(\frac{\lambda_1 \lambda_2}{\lambda_3}-\frac{1}{6}\frac{\lambda_2^3}{\lambda_3^2}\right).
\end{split}
\end{equation}
In particular,
\begin{equation}
L\left(\frac{u}{\sqrt{q}},\chi\right) = 1+\frac{\sum_{x \in \FF_q} \chi_q(\lambda_1 x - \frac{1}{2}\lambda_2 x^2 +\frac{1}{3}\lambda_3 x^3)}{\sqrt{q}} u + \chi_q\left(\frac{\lambda_1 \lambda_2}{\lambda_3}-\frac{1}{6}\frac{\lambda_2^3}{\lambda_3^2}\right) u^2.
\end{equation}
Let $c \in \mathbb{F}_p$ such that $\omega_{p}^c=\chi_q\left(\frac{\lambda_1 \lambda_2}{2\lambda_3} - \frac{1}{12}\frac{\lambda_2^3}{\lambda_3^2}\right)$. A short calculation reveals that
\begin{equation}
\begin{split}
L\left(\frac{u}{\sqrt{q}\omega_{p}^c},\chi\right) &= 1+\frac{\sum_{x \in \FF_q} \chi_q(\frac{1}{3}\lambda_3\left(x-\frac{1}{2}\frac{\lambda_2}{\lambda_3}\right)^3 + \left( \lambda_1 - \frac{\lambda_2^2}{4\lambda_3}\right) \left(x-\frac{1}{2}\frac{\lambda_2}{\lambda_3} \right))}{\sqrt{q}}u +u^2\\
&=1+\frac{\sum_{t \in \FF_q} \chi_q(\frac{1}{3}\lambda_3 t^3 + \left( \lambda_1 - \frac{\lambda_2^2}{4\lambda_3}\right) t)}{\sqrt{q}}u+u^2.
\end{split}
\end{equation}
This establishes the lemma with $a=\frac{\lambda_3}{3}$, $b=\lambda_1 - \frac{\lambda_2^2}{4\lambda_3}$, as needed.
\end{proof}
We need the following lemma, which is essentially a special case of a result of van der Geer and van der Vlugt \cite[Thm.~13.7]{van1992}. 
\begin{lem}\label{artinunity} 
Let $p=5$ and $q=p^r$. Let $\chi \in G(R_{3,q}) \setminus G(R_{2,q})$. Then the roots of $L(\frac{u}{\sqrt{q}},\chi)$ are roots of unity of order dividing $60$.
\end{lem}
\begin{proof}
Let $a,b \in \FF_{q}$, $a \neq 0$. By Lemma \ref{lstruct3}, it suffices to show that
\begin{equation}\label{eq:sviachi5}
S=\sum_{x \in \FF_{5^r}} \chi_{5^r}(ax^3+bx)
\end{equation}
is of the form $5^{r/2}(\omega_{60}^i+\omega_{60}^{-i})$ for some $i$. Van der Geer and van der Vlugt proved that the Artin-Schreier curve
\begin{equation}
C_{a,b}:\, y^5-y = ax^6+bx^2
\end{equation}
is supersingular, which implies that \cite[Lem.~3.8.4]{katz2005}
\begin{equation}
\widetilde{C_{a,b}}:\, y^5-y = ax^3+bx
\end{equation}
is also supersingular. Supersingularity of $\widetilde{C_{a,b}}$ implies that the zeros of its zeta function
\begin{equation}
Z_{\widetilde{C_{a,b}}}(u) = \frac{\prod_{i=1}^{4} L_{i}(u)}{(1-qu)(1-u)}
\end{equation}
are roots of unity divided by $\sqrt{q}$, where $L_{j}(u) = \exp( \sum_{i \ge 1} ( \sum_{x \in \FF_{q^i}} \chi_{q^i}(j (ax^3+bx) ) )u^i/i)$.
The linear coefficient of $L_1$ is $S$, $\deg L_1=2$, and the same argument as in Lemma \ref{lstruct3} shows that the leading coefficient of $L_1$ is $q$. Thus $\frac{S}{\sqrt{q}}$ is a sum of two roots of unity, the roots of $L_1(\frac{u}{\sqrt{q}})$:
\begin{equation}\label{eq:sviaomega}
S = -5^{r/2} (\omega+\omega^{-1}).
\end{equation}
By \eqref{eq:sviachi5} and \eqref{eq:sviaomega}, $\omega+\omega^{-1}$ is a (real) element of $\QQ(\omega_5,\sqrt{5})=\QQ(\omega_5)$. If $\omega$ is a primitive root of unity of order $d$, then comparing degrees in the inclusions $\mathbb{Q} \subseteq \mathbb{Q}(\omega+\omega^{-1}) \subseteq \QQ(\omega_5)$ we obtain $\phi(d)\mid 2\phi(5)=8$. A finite check shows that we must have $d \mid 60$, as needed, or $8 \mid d$, which contradicts the inclusion since $\QQ(\omega_{8}+\omega_{8}^{-1}) = \QQ(\sqrt{2}) \nsubseteq \QQ(\omega_5)$.
\end{proof}
\subsection{Proof of Theorem \ref{60period}}
Let $q=5^r$. From Lemma \ref{lemformpsi} we may write
\begin{equation}
f_q(n,t_1,t_2,t_{3})= \frac{-1}{q^{3}} \sum_{\chi_0 \neq \chi \in G(R_{3})} \overline{\chi}(t_1,t_2,t_3) \sum_{j=1}^{\deg L(u,\chi)} \left(\frac{\gamma_j(\chi)}{\sqrt{q}} \right)^n.
\end{equation}
Fix $A=(t_1,t_2,t_3) \in \mathcal{M}_q / R_{3,q}$. The group $\mathcal{M}_q / R_{3,q}$ has exponent $5$. Hence, for any $\chi \in G(R_{3})/G(R_1)$, the number $\overline{\chi}(A)$ is a root of unity of order (dividing) $5$. Thus we may write
\begin{equation}
\overline{\chi}(A) = \omega_{5}^{a_{\chi,A}},
\end{equation}
where $a_{\chi,A}$ is an integer satisfying $0 \le a_{\chi,A} < 5$.  For any $0 \le i < 60$, let
\begin{equation}
F_i(x)= \frac{-1}{q^{3}} \sum_{\chi_0 \neq \chi \in G(R_{3})} \sum_{\substack{1 \le j \le \deg L(u,\chi):\\ \frac{\gamma_j(\chi)}{\sqrt{q}}= \omega_{60}^i}} x^{a_{\chi,A}}.
\end{equation}
Set $m_1= 60$, $m_2=5$ and define the following function from $\mathbb{Z}^2$ to $\mathbb{C}$:
\begin{equation}
S(k_1,k_2) = \sum_{i=1}^{m_1} \omega_{m_1}^{ik_1} F_i(\omega_{m_2}^{k_2}).
\end{equation}
Note that
\begin{equation}
\begin{split}
S(k_1,k_2) &= \frac{-1}{q^3} \sum_{i=1}^{60} \sum_{\chi_0 \neq \chi \in G(R_3)} \sum_{\substack{1 \le j \le \deg L(u,\chi):\\  \frac{\gamma_j(\chi)}{\sqrt{q}} = \omega_{60}^i }} \omega_{60}^{ik_1} \omega_{5}^{k_2 \cdot a_{\chi,A}} \\
&= \frac{-1}{q^3} \sum_{\chi_0 \neq \chi \in G(R_3)}  \overline{\chi}(A^{k_2}) \sum_{1 \le j \le \deg L(u,\chi)} (\frac{\gamma_j(\chi)}{\sqrt{q}})^{k_1} \\
& = f_q\left(k_1, k_2 t_1, k_2 t_2 + \binom{k_2}{2} t_1^2, k_2 t_3 +k_2(k_2-1)t_1t_2 +\binom{k_2}{3} t_1^3\right).
\end{split}
\end{equation}
Applying Lemma \ref{symrootlior} to $S$, we deduce that \eqref{periodicthm5}, \eqref{sym15}, \eqref{sym25} hold. \qed
\subsection{Proof of Corollary \ref{cor:char560}}\label{sec:proofformula60}
Let $q=5^r$. Let $t_1,t_2,t_3 \in \FF_q$. For any $\chi_0 \neq \chi \in G(R_3)$ and any inverse root $\gamma$ of $L(u,\chi)$, Proposition \ref{sym2} and Lemma \ref{artinunity} tell us that $\frac{\gamma}{\sqrt{q}}$ is a  root of unity of order dividing $60$. From this point we proceed as in \S\ref{sec:proofformula24}, with 24 replaced with 60. \qed
\section{A non-periodicity result}
\subsection{Auxiliary lemmas}
\begin{lem}\label{rootsunitylemma}
Let $\alpha_1,\ldots,\alpha_m \in \mathbb{C}$ be non-zero complex numbers. Define the following sequence:
\begin{equation}
a_n = \sum_{i=1}^{m}\alpha_i^n, \qquad (n\ge 1).
\end{equation}
The sequence $\{ a_n \}_{n \ge 1}$ is periodic in $n$ if and only if all the $\alpha_i$-s are roots of unity.
\end{lem}
\begin{proof}
This follows immediately from two observations: Firstly, the generating function $\sum_{n \ge 1} a_nx^n = \sum_{i=1}^{m} \frac{\alpha_ix}{1-\alpha_ix}$ has poles at the $\alpha_i^{-1}$-s, and at them only. Secondly, the generating function of a sequence with period $N$ has poles at some subset of the roots of unity of order dividing $N$, and vice versa.
\end{proof}
\begin{lem}\label{quadreslem}
Let $p$ be a prime $\ge 7$. Let $j \in \FF_p^{\times}$ and $\varepsilon_1, \varepsilon_2 \in \{ \pm 1\}$. Then there exists $i \in \FF_p$ such that
\begin{equation}\label{existsi}
\left( \frac{i-j}{p} \right) = \varepsilon_1, \qquad \left( \frac{i+j}{p} \right) = \varepsilon_2,
\end{equation}
where $\left(\frac{\bullet}{p} \right)$ is the Legendre symbol modulo $p$. For $p=5$ it can happen that \eqref{existsi} has no solution in $i$.
\end{lem}
\begin{proof}
We begin with the case $p \ge 7$. Consider the following sum:
\begin{equation}
S(\varepsilon_1,\varepsilon_2) = \sum_{i \in \FF_p} \left(1+\varepsilon_1 \left(\frac{i-j}{p} \right) \right) \left(1+\varepsilon_2 \left(\frac{i+j}{p} \right) \right).
\end{equation}
Each term of the sum $S(\varepsilon_1,\varepsilon_2)$ has value in $\{0,2,4\}$. By counting how often each value is achieved, we find that
\begin{equation}\label{directcalcleg}
S(\varepsilon_1,\varepsilon_2) = 4 \cdot \# \{ i \in \FF_p : \left(\frac{i-j}{p}\right) = \varepsilon_1, \left( \frac{i+j}{p} \right) = \varepsilon_2\}+2 \cdot 1_{\left( \frac{2j}{p} \right) = \varepsilon_2} + 2 \cdot 1_{\left( \frac{-2j}{p} \right) = \varepsilon_1} .
\end{equation}
On the other hand, we may expand $S(\varepsilon_1,\varepsilon_2)$ and get that
\begin{equation}\label{expandcalcleg}
S(\varepsilon_1,\varepsilon_2) = p + \varepsilon_1  \sum_{i \in \FF_p} \left(\frac{i-j}{p}\right) + \varepsilon_2 \sum_{i \in \FF_p} \left(\frac{i+j}{p} \right)+ \varepsilon_1 \varepsilon_2 \sum_{i \in \FF_p}\left( \frac{i^2-j^2}{p} \right) = p+0+0-\varepsilon_1 \varepsilon_2.
\end{equation}
Comparing \eqref{directcalcleg} and \eqref{expandcalcleg}, we get that
\begin{equation}
4 \cdot \# \{ i \in \FF_p : \left(\frac{i-j}{p}\right) = \varepsilon_1, \left( \frac{i+j}{p} \right) = \varepsilon_2\} \ge p-\varepsilon_1 \varepsilon_2 -2 -2 \ge p-5 > 0,
\end{equation}
which shows that there exists $i \in \FF_p$ as required. We conclude the proof by observing that for $p=5$, there is no $i \in \FF_p$ such that
\begin{equation}
\left( \frac{i-2}{5} \right)=1, \qquad \left( \frac{i+2}{5} \right)=1. 
\end{equation}
holds.
\end{proof}
\begin{lem}\label{p7notss}
Let $p \ge 7$ be a prime. Let $\chi \in G(R_{3,p}) \setminus G(R_{2,p})$, and let $\lambda_1,\lambda_2,\lambda_3 \in \FF_q$ be the constants associated with $\chi$ in Lemma \ref{aslem}. If $p \equiv 2 \bmod 3$ and $\lambda_2^2-4\lambda_1 \lambda_3=0$, then the roots of $L(\frac{u}{\sqrt{p}},\chi)$ are roots of unity. Otherwise, both roots of $L(\frac{u}{\sqrt{p}}, \chi)$ are not roots of unity.
\end{lem}
\begin{proof}
From Lemma \ref{lstruct3} we have
\begin{equation}
L\left(\frac{u}{\sqrt{p}\omega_{p}^c},\chi\right) = 1 + \frac{\sum_{x \in \FF_p} \chi_{p}(ax^3+bx)}{\sqrt{p}} + u^2
\end{equation}
for $a=\frac{\lambda_3}{3} \in \FF_p^{\times}$, $b=\lambda_1- \frac{\lambda_2^2}{4\lambda_3}$ and some $c \in \FF_p$. We first assume that $p \equiv 2 \bmod 3$ and that $\lambda_2^2-4\lambda_1 \lambda_3=0$. This implies that $b=0$ and that the map $x\mapsto ax^3$ is a permutation of $\FF_p$. Thus, $\sum_{x\in \FF_p} \chi_p(ax^3+bx) = 0$ and the roots of $L(\frac{u}{\sqrt{p}})$ are both roots of unity.

From now on we assume that either $p \equiv 1 \bmod 3$ or that $\lambda_1 - \frac{\lambda_2^2}{4\lambda_3} \neq 0$ holds. Assume for contradiction's sake that some root of $L(\frac{u}{\sqrt{p}})$ is a root of unity. The second root of $L(\frac{u}{\sqrt{p}})$ must also be a root of unity, since the product of the two roots is a root of unity. Let $\omega_n^i, \omega_n^{-i}$ be the roots of $L(\frac{u}{\sqrt{p}\omega_{p}^c})$, where $(i,n)=1$. Consider 
\begin{equation}\label{eq:sexpress}
S=\frac{\sum_{x \in \FF_p} \chi_p(ax^3+bx)}{\sqrt{p}} = -(\omega_n^i +\omega_n^{-i}).
\end{equation}
The quadratic Gauss sum $g_p=\sum_{a \in \FF_p} \left( \frac{a}{p} \right) \omega_p^{a}$ satisfies $g_p \in \mathbb{Q}(\omega_p)$ and $g_p^2 = (-1)^{\frac{p-1}{2}}p$ \cite[Prop.~6.3.2]{ireland1990}. Thus we have 
\begin{equation}\label{sfield}
S \in \mathbb{Q}(\omega_p, \sqrt{p}) = \mathbb{Q} (\omega_{t p}), \qquad t=\begin{cases} 1 & \mbox{if $p \equiv 1 \bmod 4$,} \\ 4 & \mbox{if $p \equiv 3 \bmod 4$.} \end{cases}
\end{equation}
Moreover, $S$ is real since $\omega_n^i + \omega_n^{-i}$ is real, hence it follows from \eqref{sfield} that in fact
\begin{equation}\label{scontainedreal}
S \in \QQ(\omega_{tp}+\omega_{tp}^{-1}).
\end{equation}
Since $S$ generates $\mathbb{Q}(\omega_n+\omega_n^{-1})$, we have from \eqref{scontainedreal} that
\begin{equation}\label{twofieldscont}
\QQ(\omega_n+\omega_n^{-1}) \subseteq \QQ(\omega_{tp}+\omega_{tp}^{-1}).
\end{equation}
We have
\begin{equation}\label{fieldcont}
\QQ(\omega_n+\omega_n^{-1}) = \QQ(\omega_n+\omega_n^{-1}) \cap \QQ(\omega_{tp}+\omega_{tp}^{-1}) \subseteq \QQ(\omega_n) \cap \QQ(\omega_{tp}) = \QQ(\omega_{(n,tp)}).
\end{equation}
Since $\QQ(\omega_n+\omega_n^{-1}) \subseteq \mathbb{R}$, \eqref{fieldcont} implies that
\begin{equation}\label{twoqreals}
\QQ(\omega_n+\omega_n^{-1}) = \QQ(\omega_{(n,tp)}+\omega_{(n,tp)}^{-1})
\end{equation}
If $n=1$ or $n=2$, \eqref{twoqreals} is trivial. If $n>2$, by comparing degrees over $\QQ$, \eqref{twoqreals} implies that
\begin{equation}\label{divdim}
\phi(n) = \phi((n,tp)),
\end{equation}
which in turn implies that $n \mid tp$. Both for $n=1,2$ and $n>2$, we have
\begin{equation}
n \mid 2(2,t)p.
\end{equation}
If $p\equiv 1 \bmod 4$ then $n \mid 2p$, hence for some $j \in \ZZ$ we have
\begin{equation}\label{p14simp}
\omega^i_n + \omega_n^{-i} = \pm(\omega_p^j + \omega_p^{-j}), \qquad \sqrt{p} = \sum_{i \in \FF_p} \left( \frac{i}{p} \right) \omega_p^i.
\end{equation}
Plugging \eqref{p14simp} in \eqref{eq:sexpress} we get
\begin{equation}\label{characterequation1}
\begin{split}
\sum_{x \in \FF_p} \chi_p(ax^3+bx) &= \pm (\omega_p^j + \omega_p^{-j})\sum_{i \in \FF_p} \left( \frac{i}{p} \right) \omega_p^i \\
&=\pm \sum_{i \in \FF_p} \left( \frac{i}{p} \right) (\omega_p^{i+j} + \omega_p^{i-j}).
\end{split}
\end{equation}
If $p \equiv 3 \bmod 4$ then $n \mid 4p$, hence for some $j,k \in \ZZ$ we have
\begin{equation}\label{p34simp}
\omega^i_n + \omega_n^{-i} = \omega_p^j \omega_4^k + \omega_p^{-j}\omega_4^{-k}, \qquad \sqrt{p} = -\omega_4 \sum_{i \in \FF_p} \left( \frac{i}{p} \right) \omega_p^i.
\end{equation}
Plugging \eqref{p34simp} in \eqref{eq:sexpress} we get
\begin{equation}\label{p3ijk}
\sum_{x \in \FF_p} \chi_p(ax^3+bx) =  \omega_4 \left( \omega_p^j \omega_4^k + \omega_p^{-j}\omega_4^{-k} \right) \sum_{i \in \FF_p} \left( \frac{i}{p} \right) \omega_p^i.
\end{equation}
If $k$ is even, \eqref{p3ijk} shows that $\omega_4 \in \QQ(\omega_p)$, a contradiction. Hence we may assume that $k$ is odd. In this case \eqref{p3ijk} is equivalent to
\begin{equation}\label{eq:characterequation2}
\begin{split}
\sum_{x \in \FF_p} \chi_p(ax^3+bx) &= \pm ( \omega_p^j - \omega_p^{-j}) \left( \sum_{i \in \FF_p} \left( \frac{i}{p} \right) \omega_p^i \right)\\
&= \pm \sum_{i \in \FF_p} \left( \frac{i}{p} \right)\left( \omega_p^{i+j} - \omega_p^{i-j} \right). 
\end{split}
\end{equation}
The only non-trivial linear relation among $\{ \omega_p^i : 0 \le i \le p-1 \}$ is that their sum is $0$. Hence, we may compare the coefficient of $\omega_p^i$ in both sides of \eqref{characterequation1} (if $p\equiv 1 \bmod 4$) or of \eqref{eq:characterequation2} (if $p \equiv 3 \bmod 4$), up to some constant $\lambda \in \mathbb{C}$. This gives us the following equality for all $i \in \FF_q$:
\begin{equation}\label{nonperiodeq}
\# \{ x \in \FF_p : ax^3 + bx = i \} = \lambda + \varepsilon_1 \left( \frac{i-j}{p} \right) + \varepsilon_2 \left( \frac{i+j}{p} \right),
\end{equation}
where $\lambda$ is some complex number and $\varepsilon_1,\varepsilon_2 \in \{ \pm 1\}$ satisfy 
\begin{equation}
\varepsilon_1 \varepsilon_2 = \begin{cases} 1 & \text{if }p \equiv 1 \bmod 4, \\ -1 & \text{if }p \equiv 3 \bmod 4.\end{cases}
\end{equation}
Summing \eqref{nonperiodeq} over all $i \in \FF_p$ we see that necessarily $\lambda = 1$. If $j \neq 0$, then Lemma \ref{quadreslem} shows that there exists $i \in \FF_p$ such that $\left( \frac{i-j}{p} \right) = -\varepsilon_1, \left( \frac{i+j}{p} \right) = -\varepsilon_2$, in which case the right hand side of \eqref{nonperiodeq} is negative while the left hand side is non-negative, which gives a contradiction. Thus, we must have $j=0$.

If $j=0$ and $p \equiv 1 \bmod 4$, then for any $i \in \FF_q$ such that $\left(\frac{i}{p} \right) = -\varepsilon_1 = -\varepsilon_2$ we get that the right hand side of \eqref{nonperiodeq} is equal to $-1$, while the left hand side is non-negative, which gives a contradiction. Thus, we must have $j=0$ and $p \equiv 3 \bmod 4$. In that case, $\varepsilon_1 = -\varepsilon_2$ and \eqref{nonperiodeq} simplifies to
\begin{equation}
\# \{ x \in \FF_p : ax^3 + bx = i \} = 1
\end{equation}
for all $i \in \FF_p$, which means that $ax^3+bx$ is a \emph{permutation polynomial} on $\FF_p$. Dickson classified permutation polynomials of degree $3$ \cite[p.~63]{dickson1958}, and in characteristic $>3$ they exist only when $p \equiv 2 \bmod 3$ and are given by $a_1(x+a_2)^3+a_3$. By assumption, either $p \equiv 1 \bmod 3$, in which case we do not have permutation polynomials, or $b \neq 0$, in which case $ax^3+bx$ is not of the form $a_1(x+a_2)^3+a_3$. We reach a contradiction, hence our assumption is wrong and none of the roots of $L(u/\sqrt{q}, \chi)$ is a root of unity.
\end{proof}
\subsection{Proof of Theorem \ref{conv2460}}
By Lemma \ref{lemformpsi} we have the following formula for $f_{q,\ell}(n)$:
\begin{equation}
f_{q,\ell}(n) = -q^{-\ell}\sum_{\chi_0 \neq \chi \in G(R_{\ell,q})} \sum_{i=1}^{\deg L(u,\chi) }\left( \frac{\gamma_i(\chi)}{\sqrt{q}} \right)^n,
\end{equation}
where $\gamma_i(\chi)$-s are the inverse roots of $L(u,\chi)$. Lemma \ref{rootsunitylemma} shows that $f_{q,\ell}(n)$ is not periodic if and only if there is some $\chi_0 \neq \chi \in G(R_{\ell,q})$ such that at least one the ratios $\gamma_i(\chi)/\sqrt{q}$ is not a root of unity. So our aim is to prove that such a character exists in all the cases of Theorem \ref{conv2460}. If $q=p^r$, then it is in fact enough to prove the existence of such a character for $q=p$, since given $\chi_0 \neq \chi \in G(R_{\ell,p})$ such that at least one the ratios $\gamma_i(\chi)/\sqrt{p}$ is not a root of unity, we may take its lift $\chi^{(r)} \in G(R_{\ell,q})$ to $\FF_q$, and \eqref{liftlform} guarantees that $L(u/\sqrt{p^r},\chi^{(r)})$ also has a root which is not a root of unity. This is what we show now.
\begin{enumerate}
\item $p=2$, $\ell \ge 4$: Lemma \ref{structlem} shows that the group $\mathcal{M}_2 / R_{4}$ is isomorphic to $(\mathbb{Z}/8\mathbb{Z}) \times (\mathbb{Z}/2\mathbb{Z})$ (the unique abelian group of order $16$ and exponent equal to $8$). It is generated, for example, by the elements $T+1$, $T^3+1$ of order $8$ and $2$, respectively. Consider the character $\chi \in G(R_{4,2}) \setminus G(R_{3,2}) \subseteq G(R_{\ell,2})$ defined by
\begin{equation}
\chi(1,0,0,0) = \omega_8, \quad \chi(0,0,1,0)=1.
\end{equation}
The relations
\begin{equation}
\begin{split}
(0,1,0,0)=(1,0,0,0)^2, & \quad(0,1,1,0) = (1,0,0,0)^2 (0,0,1,0), \\
(1,0,1,0)=(1,0,0,0)^5 (0,0,1,0), & \quad (1,1,0,0) = (1,0,0,0)^7 (0,0,1,0),\\
(1,1,1,0)=(1,0,0,0)^3 &
\end{split}
\end{equation}
imply that
\begin{equation}
L\left(\frac{u}{\sqrt{2}},\chi\right)=1+\frac{1+\omega_8}{\sqrt{2}} u + \frac{1+\omega_8}{\sqrt{2}}u^2 + \omega_8 u^3 = (1+iu)(1 +(\frac{1+\omega_8}{\sqrt{2}}-i)u - \omega_8 i u^2),
\end{equation}
which has two roots which are not even algebraic integers.
\item $p=3$, $\ell \ge 3$: Lemma \ref{structlem} shows that the group $\mathcal{M}_3 / R_{3}$ is isomorphic to $(\mathbb{Z}/9\mathbb{Z}) \times (\mathbb{Z}/3\mathbb{Z})$ (the unique abelian group of order $27$ and exponent equal to $9$). It is generated, for example, by the elements $T+1$ and $T^2+1$ of order $9$ and $3$, respectively. Consider the character $\chi \in G(R_{3,3}) \setminus G(R_{2,3}) \subseteq G(R_{\ell,3})$ defined by
\begin{equation}
\chi(1,0,0) = \omega_9, \quad \chi(0,1,0)=1.
\end{equation}
The relations
\begin{equation}
\begin{split}
(2,0,0)=(1,0,0)^8(0,1,0)^2, & \quad(0,2,0) = (0,1,0)^2, \\
(1,1,0)=(1,0,0)^7 (0,1,0), & \quad (1,2,0) = (1,0,0)^4 (0,1,0)^2,\\
(2,1,0)=(1,0,0)^2 &, \quad (2,2,0)=(1,0,0)^5 (0,1,0)
\end{split}
\end{equation}
imply that
\begin{equation}
L\left(\frac{u}{\sqrt{3}},\chi\right)=1+\frac{1+\omega_9+\overline{\omega_9}}{\sqrt{3}} u + u^2,
\end{equation}
which has two roots which are not even algebraic integers
\item $p=5$, $\ell \ge 4$: By taking $\ell=4,\lambda_1=\lambda_2=\lambda_3=0,\lambda_4 = 1$ in Lemma \ref{aslem}, we get that the following function is a character of $R_{4,5}$:
\begin{equation}
\chi(a_1,a_2,a_3,a_4) = \omega_5^{a_1^4 +a_1^2 a_2 -a_1 a_3 -\frac{a_2^2}{2} + a_4}.
\end{equation}
The $L$-function of $\chi$ is
\begin{equation}
L\left(\frac{u}{\sqrt{5}},\chi\right) = 1+ \frac{1+4\omega_5}{\sqrt{5}}u - \frac{1+4\omega_5^{4}}{\sqrt{5}} u^2 - u^3 = (1+u)(1+\frac{1-\sqrt{5}	+4\omega_5}{\sqrt{5}}u-u^2),
\end{equation}
which has two roots which are not roots of unity.
\item $p \ge 7$, $\ell \ge 3$: This case is treated in Lemma \ref{p7notss}, which shows that if $p \equiv 1 \bmod 3$ we may take any $\chi \in G(R_{3,p}) \setminus G(R_{2,p})$, and if $p \equiv 2 \bmod 3$ then most choices of $\chi \in G(R_{3,p}) \setminus G(R_{2,p})$ work. 
\end{enumerate}
\qed
\section{Supersingular curves}\label{sec:cyclotomic}
Let $q=p^r$ be a prime power. A curve $C$ defined over $\FF_q$ is called supersinguar if its Jacobian is isogenous to a product of supersingular elliptic curves. This condition is known to be the same as saying that the roots of the zeta function of $C$,
\begin{equation}
Z_C(u) = \exp\left( \sum_{n \ge 1} \frac{\# C(\FF_{q^n}) u^n}{n}\right),
\end{equation}
are of the form $\sqrt{q}$ times roots of unity.

We use an analogue of cyclotomic fields in finite fields, due to Carlitz \cite{carlitz1938} and Hayes \cite{hayes1974}, and we follow Rosen \cite[Ch.~12]{rosen2002} in our notation. For a given $\ell \ge 1$, we consider the following fields. Let $K_{\ell}=\FF_q(T)(\Lambda_{T^{-\ell-1}})$ be the abelian extension of $\FF_q(T)$ obtained by adjoining the $T^{-\ell-1}$--torsion elements of the Carlitz module $\Lambda$. The extension $K_{\ell}/\FF_q(T)$ is totally ramified at the prime at infinity, and unramified everywhere else. We have $[K_{\ell}:\FF_q(T)]=q^{\ell}(q-1)$, and $K_{\ell}$ has a unique subfield $L_{\ell}$ with $[L_{\ell}:\FF_q(T)]=q^{\ell}$, which is called the maximal real subfield of $K_{\ell}$. The Dedekind zeta function of $L_{\ell}$ factors as the product of $L(u,\chi)$ for $\chi \in G(R_{\ell,q})$, after completing these $L$-functions at the prime at infinity. This is because the Dirichlet characters $G(R_{\ell,q})$ can be considered as the representations of $\mathrm{Gal}(L_{\ell}/\FF_q(T))$. The field $L_{\ell}$ is the function field of a certain smooth projective curve, which we denote by $C_{\ell}$. The Dedekind zeta function of $L_{\ell}$ is the same the zeta function of $C_{\ell}$, so that
\begin{equation}\label{prod:zeta}
Z_{C_{\ell}}(u) = \prod_{\chi \in G(R_{\ell,q})} L^{*}(u,\chi).
\end{equation}
Here $L^{*}$ is the completed $L$-function. As $\chi \in G(R_{\ell,q})$ vanishes on the prime at infinity unless $\chi$ is trivial, we have 
\begin{equation}
L^{*}(u,\chi) = \begin{cases} L(u,\chi) & \mbox{if $\chi_0 \neq \chi \in G(R_{\ell})$,} \\ L(u,\chi)(1-u)^{-1} & \mbox{if $\chi = \chi_0$.}\end{cases}  
\end{equation}
From Lemma \ref{lemformpsi} and \eqref{prod:zeta}, it follows that 
\begin{equation}
\psi_q(n,\underbrace{0,\ldots,0}_{\ell}) = \frac{\#C_{\ell}(\FF_{q^n})-1}{q^{\ell}}.
\end{equation}
The number
\begin{equation}
\pi_q(n,\underbrace{0,\ldots,0}_{\ell})
\end{equation}
can be interpreted as the number of primes in $\FF_q[T]$ of degree $n$ that split completely in $L_{\ell}$. Let $g$ be the genus of $C_{\ell}$. As $Z_{C_{\ell}}(u) = P_{C_{\ell}}(u)/((1-qu)(1-u))$ where $P_{C_{\ell}}$ is a polynomial of degree $2g$, and $\deg L(u,\chi) = k$ if $\chi \in G(R_{k,q}) \setminus G(R_{k-1,q})$, \eqref{prod:zeta} implies that
\begin{equation}
2g = \sum_{k=1}^{\ell} (k-1) (q^k - q^{k-1}) = (\ell-1)q^{\ell} - (q+q^2+\ldots +q^{\ell-1}).
\end{equation} 
Note that $C_{1}$ is of genus $0$. By Proposition \ref{sym2}, $C_{2}$ is supersingular. By Theorem \ref{conv2460}, $C_{3}$ is supersingular if and only if $p \in \{2,5\}$, and $C_{\ell}$ is not supersingular for $\ell \ge 4$. 

Finally we explain how to write an equation for an affine model of $C_{\ell}$. Define $g_n(X) \in \FF_q(T)[X]$ recursively by $g_{0}(X)=1$, $g_1(X)=X^q+\frac{X}{T}$, $g_{n+1}(X) = g_{n}(X)^q+\frac{g_n(X)}{T}$. The irreducible polynomial generating $K_{\ell}$ is $\frac{g_{\ell+1}(X)}{g_{\ell}(X)}$, which is a function field analogue of a cyclotomic polynomial. Induction shows that $g_i(X) = X h_i(X^{q-1})$ for a unique $h_i\in \FF_q(T)[X]$, and so $\frac{g_{\ell+1}(X)}{g_{\ell}(X)} = \frac{h_{\ell+1}}{h_{\ell}}(X^{q-1})$. As $[K_{\ell}:L_{\ell}]=q-1$, we find that $L_{\ell}$ is the splitting field of $\frac{h_{\ell+1}}{h_{\ell}}(X)$. So $\frac{h_{\ell+1}}{h_{\ell}}(X) = 0$ is an affine model for $C_{\ell}$. For instance, a short calculation gives the following affine model for $C_{1}$:
\begin{equation}
\sum_{i=1}^{q} X^{i}\left(\frac{-1}{T}\right)^{q-i} +\frac{1}{T} = 0.
\end{equation}
\section*{Acknowledgments}
I thank Ze'ev Rudnick for useful discussions, Lior Bary-Soroker for suggesting Lemma \ref{symrootlior} and an anonymous referee for constructive comments. I was supported by the European Research Council under the European Union's Seventh Framework Programme (FP7/2007-2013) / ERC grant agreement n$^{\text{o}}$ 320755.
\appendix
\section{Computing \texorpdfstring{$\pi_q$}{pi q} from \texorpdfstring{$\psi_q$}{psi q}}\label{secmob}
\subsection{M\"obius inversion}
We explain how exact formulas for $\psi_q(n,t_1,\ldots,t_{\ell})$ yield exact formulas for $\pi_q(n,t_1,\ldots, t_{\ell})$. We introduce the following notation. For a given arithmetic function
\begin{equation}
\alpha \colon \mathcal{M}_{q} \to \mathbb{C}
\end{equation}
and a positive integer $\ell$, we attach a summatory function 
\begin{equation}
S_{\alpha, \ell}\colon \mathbb{N}_{>0} \times \mathcal{M}_q / R_{\ell} \to \mathbb{C}
\end{equation}
defined by
\begin{equation}
S_{\alpha, \ell}(n,A) = \sum_{f \in \mathcal{M}_{n,q}, f\equiv A \bmod R_{\ell}} \alpha(f).
\end{equation}
In this notation,
\begin{equation}
\begin{split}
\psi_q(n,t_1,\ldots, t_{\ell}) &= S_{\Lambda_q, \ell}(n,(t_1,\ldots,t_{\ell})), \\
\pi_q(n,t_1,\ldots, t_{\ell}) &= S_{1_{\mathcal{P}_q},\ell}(n,(t_1,\ldots, t_{\ell})).
\end{split}
\end{equation}
\begin{lem}\label{lemmob}
Let $n,\ell$ be two positive integers. Let $A \in \mathcal{M}_{n,q}$. We have
\begin{equation}\label{idenmob}
S_{1_{\mathcal{P}_q},\ell}(n,A) = \frac{\sum_{d \mid n} \mu(d) \sum_{B\in \mathcal{M}_q / R_{\ell} : B^d \equiv A \bmod R_{\ell}} S_{\Lambda_q, \ell}(\frac{n}{d},B)}{n}.
\end{equation}

\end{lem}
\begin{proof}
We may write $S_{1_{\mathcal{P}_q},\ell}(n,A)$, $S_{\Lambda_q, \ell}(n,A)$ as sums over primes of degree dividing $n$:
\begin{equation}
\begin{split}
S_{1_{\mathcal{P}_q}, \ell}(n,A) &= \sum_{d \mid n} \sum_{P \in \mathcal{P}_q, \deg P =d} 1_{d = n} \cdot 1_{P \equiv A \bmod R_{\ell}},\\
S_{\Lambda_q, \ell}(n,A) &= \sum_{d \mid n}\sum_{P \in \mathcal{P}_q, \deg P =d} d \cdot 1_{P^{n/d} \equiv A \bmod R_{\ell}}.
\end{split}
\end{equation}
Fix $i \mid n$ and a prime $P \in \mathcal{P}_q$ of degree $i$. To establish Lemma \ref{lemmob}, it suffices to show that $P$ contributes the same weight to each side of \eqref{idenmob}. The contribution of $P$ to the left hand side of \eqref{idenmob} is
\begin{equation}\label{cont1}
1_{i = n} \cdot 1_{P \equiv A \bmod R_{\ell}}.
\end{equation}
The contribution of $P$ to the right hand side is
\begin{equation}
\frac{\sum_{d \mid n:\, i \mid \frac{n}{d}} \mu(d) \sum_{B \in \mathcal{M}_q/ R_{\ell}: B^d \equiv A \bmod R_{\ell}} i \cdot 1_{P^{\frac{n}{di}} \equiv B \bmod R_{\ell}}}{n},
\end{equation}
which may be simplified as
\begin{equation}\label{cont2}
\begin{split}
\frac{i}{n} \sum_{d \mid n:\, i \mid \frac{n}{d}} \mu(d) \sum_{B \in \mathcal{M}_q/ R_{\ell}:B^d \equiv A \bmod R_{\ell}} 1_{P^{\frac{n}{di}} \equiv B \bmod R_{\ell}}&=\frac{i}{n} \sum_{d \mid \frac{n}{i}} \mu(d) 1_{P^{\frac{n}{i}} \equiv A \bmod R_{\ell}}\\
&= \frac{i}{n} \cdot  1_{P^{\frac{n}{i}} \equiv A \bmod R_{\ell}}  \sum_{d \mid \frac{n}{i}} \mu(d) \\
&= \frac{i}{n} \cdot 1_{P^{\frac{n}{i}} \equiv A \bmod R_{\ell}} \cdot 1_{i = n} = 1_{i=n} \cdot 1_{P \equiv A \bmod R_{\ell}}.
\end{split}
\end{equation}
Here we have used the property $\sum_{d \mid m} \mu(d) = 1_{m=1}$ of the M\"obius function. The expressions \eqref{cont1} and \eqref{cont2} coincide, as needed. 
\end{proof}
\begin{remark}
Similarly, it may be shown that
\begin{equation}
S_{\Lambda_q,\ell}(n,A) = \sum_{d \mid n}\frac{n}{d}  \sum_{B\in \mathcal{M}_q / R_{\ell} : B^d \equiv A \bmod R_{\ell}} S_{1_{\mathcal{P}_q},\ell}(\frac{n}{d},B).
\end{equation}
\end{remark}
\subsection{Solving equations in \texorpdfstring{$\mathcal{M}_q /R_{\ell}$}{Mq R l}}
In order to make use of Lemma \ref{lemmob}, we describe the following sets. Given $A \in \mathcal{M}_q /R_{\ell}$ and $d \ge 1$, let
\begin{equation}
S_{A,d,\ell,q} = \{ B \in \mathcal{M}_q / R_{\ell} : B^d \equiv A \bmod R_{\ell} \}.
\end{equation}
\begin{lem}\label{solvingpowerbda}
Let $q=p^r$ be a prime power. Let $\ell \ge 1$ and $A=(a_1,\ldots,a_{\ell}) \in \mathcal{M}_q / R_{\ell}$. Let $d \ge 1$ and write it as $d = d' \cdot k$ with $d'$ the largest power of $p$ dividing $d$. Let $e$ be the smallest non-negative integer such that $q^e \ge d'$. If $a_i \neq 0$ for some $d' \nmid i$, then $S_{A,d,\ell,q}=\varnothing$. Otherwise,
\begin{equation}\label{sathmval}
S_{A,d,\ell,q} =  \{ (b_1, \ldots, b_{\ell}) \in \mathcal{M}_q / R_{\ell} \mid \forall 1\le j \le \lfloor \frac{\ell}{d'} \rfloor : b_{j} = \sum\limits_{ \sum\limits_{i=1}^{ \lfloor \frac{\ell}{d'} \rfloor } i c_i = j} \prod\limits_{i=1}^{\lfloor \frac{\ell}{d'} \rfloor} a_{i d'}^{\frac{q^{e}}{d'} c_i} \binom{\frac{1}{k}}{\sum c_i} \binom{\sum c_i}{c_1,c_2,\ldots,c_{\lfloor \frac{\ell}{d'} \rfloor}} \}.
\end{equation}
\end{lem}
\begin{proof}
Let $B=(b_1,\ldots, b_{\ell}) \in \mathcal{M}_q / R_{\ell}$. The equation $B^d \equiv A \bmod R_{\ell}$ may be written as
\begin{equation}\label{ikeq}
(B^{d'})^k \equiv A \bmod R_{\ell}.
\end{equation}
By Lemma \ref{structlem}, the exponent of $\mathcal{M}_q / R_{\ell}$ is $p^{e_1}$ where $e_1$ is the smallest integer satisfying $p^{e_1} \ge \ell+1$. Hence we have $g^j \equiv g \bmod R_{\ell}$ whenever $j \equiv 1 \bmod p^{e_1}$. In particular, by Euler's theorem, $g^{k^{\varphi(p^{e_1})}} \equiv g \bmod R_{\ell}$.
By raising both sides of \eqref{ikeq} to the $k^{\varphi(p^{e_1})-1}$-th power, we find that \eqref{ikeq} is equivalent to
\begin{equation}\label{pied}
B^{d'} \equiv A^{k^{\varphi(p^{e_1})-1}} \bmod R_{\ell}.
\end{equation}
As we are in characteristic $p$, $(\sum a_i T^i)^{d'} = \sum a_i^{d'} T^{id'}$. In particular, the $j$-th coordinate of $B^{d'}$ is 0 if $d' \nmid j$, and is $b^{d'}_{\frac{j}{d'}}$ otherwise. Since $k^{\varphi(p^{e_1})-1}$ is coprime to the order of the group $\mathcal{M}_q/R_{\ell}$, it follows that $A$ is a $d'$-th power if and only if $A^{k^{\varphi(p^{e_1})-1}}$ is a $d'$-th power. Hence, if $a_{j} \neq 0$ for some $j$ not divisible by $d'$, then \eqref{pied} is not solvable in $B$. From now on we assume that $a_j =0$ for any $j$ not divisible by $d'$, and define
\begin{equation}
\tilde{A} = (a_{d'}^{q^{e}/d'}, a_{2 d'}^{q^{e}/d'}, \ldots, a_{\lfloor \frac{\ell}{d'} \rfloor d' }^{q^{e}/d'}, 0, \ldots, 0 ) \in \mathcal{M}_q/ R_{\ell},
\end{equation}
which satisfies $\tilde{A}^{d'} = A$. We define the following truncations of $\tilde{A}$ and $B$:
\begin{equation}
\tilde{A}_{\text{trunc}} = (a_{d'}^{q^{e}/d'}, a_{2 d'}^{q^{e}/d'}, \ldots, a_{\lfloor \frac{\ell}{d'} \rfloor d'}^{q^{e}/d'}), \quad B_{\text{trunc}} = (b_1, b_2, \ldots ,b_{\lfloor \frac{\ell}{d'} \rfloor}) \in \mathcal{M}_q / R_{ \lfloor \frac{\ell}{d'} \rfloor}.
\end{equation}
Equation \eqref{pied} is equivalent to
\begin{equation}\label{trunceq}
B_{\text{trunc}} \equiv \tilde{A}_{\text{trunc}}^{k^{\varphi(p^{e_1})-1}} \bmod R_{ \lfloor \frac{\ell}{d'} \rfloor}.
\end{equation}
Equation \eqref{trunceq} is the same as saying that $b_1, \ldots, b_{\lfloor \frac{\ell}{d'} \rfloor}$ assume the following values:
\begin{equation}\label{eqsbella}
b_{j} = \sum_{ \sum_{i=1}^{ \lfloor \frac{\ell}{d'} \rfloor } i c_i = j} \prod_{i=1}^{\lfloor \frac{\ell}{d'} \rfloor} a_{i d'}^{\frac{q^{e}}{d'} c_i} \binom{k^{\varphi(p^{e_1})-1}}{\sum c_i} \binom{\sum c_i}{c_1,c_2,\ldots,c_{\lfloor \frac{\ell}{d'} \rfloor}} \in \mathbb{F}_q.
\end{equation}
To establish \eqref{sathmval} from \eqref{eqsbella}, it remains to show that $\binom{k^{\varphi(p^{e_1})-1}}{\sum c_i}$ in \eqref{eqsbella} may be replaced with $\binom{\frac{1}{k}}{\sum c_i}$. This follows from Lucas' Theorem on the values of binomial coefficients modulo primes \cite{lucas1878}, which tells us that the function
\begin{equation}
f\colon  \mathbb{Q} \cap \mathbb{Z}_p  \to \mathbb{F}_q, \quad x \mapsto \binom{x}{\sum c_i}.
\end{equation}
depends only on the first $e_3$ $p$-adic digits of its input, where $e_3$ is the smallest integer satisfying $p^{e_3} > \sum c_i$. As $\sum c_i \le \sum i c_i = j\le \ell$, we have $e_1 \ge e_3$ and Euler's theorem implies that $k^{\varphi(p^{e_1})-1} \equiv \frac{1}{k} \bmod p^{e_3}$. Thus, $\binom{k^{\varphi(p^{e_1})-1}}{\sum c_i} \equiv  \binom{\frac{1}{k}}{\sum c_i} \bmod p$ as needed.
\end{proof}
From Lemmas \ref{lemmob} and \ref{solvingpowerbda}, we immediately obtain the following proposition.
\begin{proposition}\label{fullmobinv}
Let $n,\ell$ be positive integers. Let $\FF_q$ be a finite field of characteristic $p$. Let $t_1,\ldots,t_{\ell} \in \FF_q$. For any divisor $d$ of $n$ not divisible by $p$, define 
\begin{equation}\label{deftj}
\forall 1 \le j \le \ell: t_{j,d}= \sum\limits_{ \sum\limits_{i=1}^{ \ell } i c_i = j} \prod\limits_{i=1}^{\ell} t_{i}^{c_i} \binom{\frac{1}{d}}{\sum c_i} \binom{\sum c_i}{c_1,c_2,\ldots,c_{\ell}} \in \FF_q.
\end{equation}
For any divisor $d$ of $n$ divisible by $p$ exactly once, define 
\begin{equation}\label{defsj}
\forall 1 \le j \le \lfloor \frac{\ell}{p} \rfloor: \tilde{t}_{j,d}= \sum\limits_{ \sum\limits_{i=1}^{ \lfloor \frac{\ell}{p} \rfloor } i c_i = j} \prod\limits_{i=1}^{\lfloor \frac{\ell}{p} \rfloor} t_{ip}^{\frac{q}{p} c_i} \binom{\frac{p}{d}}{\sum c_i} \binom{\sum c_i}{c_1,c_2,\ldots,c_{\lfloor \frac{\ell}{p} \rfloor}} \in \FF_q.
\end{equation}
Then we have
\begin{equation}\label{idenmobexplicit}
\begin{split}
\pi_q(n,t_1,\ldots, t_{\ell}) &= \frac{1}{n} \sum\limits_{d:\, p \nmid d \mid n} \mu(d) \psi_q( \frac{n}{d}, t_{1,d}, t_{2,d}, \ldots, t_{\ell,d}) \\
& \quad + 1_{\forall i \le \ell \text{ s.t. }p\nmid i: \, t_i = 0 }\cdot \frac{1}{n} \sum\limits_{d:\, p \mid d \mid n,\, p^2 \nmid d} \mu(d) \psi_q( \frac{n}{d}, \tilde{t}_{1,d}, \tilde{t}_{2,d}, \ldots, \tilde{t}_{\lfloor \frac{\ell}{p} \rfloor, d}).
\end{split}
\end{equation}
\end{proposition}
\subsection{The case \texorpdfstring{$\ell=3$}{l=3}}
\begin{cor}\label{inv32}
Let $q=p^r$ be a prime power and let $n \in \mathbb{N}_{>0}$. Let $t_1,t_2,t_3 \in \FF_q$.
\begin{enumerate}
\item If $p \ge 5$ then 
\begin{equation}
\begin{split}
\pi_q(n,t_1,t_2,t_3) &= \frac{\sum_{d:\, p \nmid d \mid n} \mu(d) \psi_q(\frac{n}{d}, \frac{t_1}{d}, \frac{t_2}{d} + \frac{1-d}{2d^2} t_1^2, \frac{t_3}{d} + \frac{1-d}{d^2} t_1 t_2 + \frac{(2d-1)(d-1)}{6d^3} t_1^3)}{n}\\
&\quad + 1_{t_1=t_2=t_3=0} \cdot \frac{\sum_{d:\,p \mid d \mid n} \mu(d) q^{\frac{n}{d}}}{n}.
\end{split}
\end{equation}
\item If $p =3$ then
\begin{equation}
\begin{split}
\pi_q(n,t_1,t_2,t_3) &= \frac{\sum_{d:\,3 \nmid d \mid n} \mu(d) \psi_q(\frac{n}{d}, \frac{t_1}{d}, \frac{t_2}{d} + \frac{1-d}{2d^2} t_1^2, \frac{t_3}{d} + \frac{1-d}{d^2} t_1 t_2 + \frac{(2d-1)(d-1)}{6d^3} t_1^3)}{n}\\
&\quad + 1_{t_1=t_2=0} \cdot \frac{\sum_{d:\, 3 \mid  d \mid n} \mu(d) q^{\frac{n}{d}-1}}{n}.
\end{split}
\end{equation}
\item If $p=2$ then
\begin{equation}
\begin{split}
\pi_q(n,t_1,t_2,t_3) &= \frac{\sum_{d \equiv 1 \bmod 4, d \mid n} \mu(d) \psi_q(\frac{n}{d}, t_1,t_2,t_3)}{n} +\frac{\sum_{d \equiv 3 \bmod 4, d \mid n} \mu(d) \psi_q(\frac{n}{d}, t_1,t_2+t_1^2,t_3+t_1^3)}{n} \\
& \quad + 1_{t_1=t_3=0} \cdot  \frac{\sum_{d:\, 2 \mid d \mid n} \mu(d) q^{\frac{n}{d}-1}}{n}.
\end{split}
\end{equation}
\end{enumerate}
\end{cor}
\begin{proof}
Let $(t_1,t_2,t_3) \in \mathcal{M}_q / R_3$. Specializing \eqref{idenmobexplicit} to $\ell=3$, we get
\begin{equation}\label{idenmob3}
\begin{split}
\pi_q(n,t_1,t_2,t_3) &= \frac{1}{n} \sum\limits_{d:\, p \nmid d \mid n} \mu(d) \psi_q( \frac{n}{d}, \frac{t_1}{d}, \frac{t_2}{d} + \frac{1-d}{2d^2} t_1^2, \frac{t_3}{d} + \frac{1-d}{d^2} t_1t_2 + \frac{(2d-1)(d-1)}{6d^3}t_1^3) \\
& \quad +  1_{\forall i \le 3 \text{ s.t. }p\nmid i: \, t_i = 0 }\cdot \frac{1}{n} \sum\limits_{d:\, p \mid d \mid n,\, p^2 \nmid d} \mu(d) \psi_q( \frac{n}{d}, \tilde{t}_{1,d}, \tilde{t}_{2,d}, \ldots, \tilde{t}_{\lfloor \frac{3}{p} \rfloor, d}).
\end{split}
\end{equation}
where $\tilde{t}_{i,\lfloor \frac{3}{p} \rfloor}$ are defined in \eqref{defsj}.
\begin{enumerate}
\item Assume that $p\ge 5$. Then the condition that $t_i=0$ if $i\le 3$ and $p \nmid i$ (appearing in the second term of \eqref{idenmob3}) is equivalent to $t_1=t_2=t_3=0$. This shows that the second term in \eqref{idenmob3} is 
\begin{equation}\label{secondsumpgreater5}
1_{t_1=t_2=t_3=0} \cdot \frac{1}{n} \sum_{p \mid d \mid n} \mu(d) \psi_q(\frac{n}{d}).
\end{equation}
From \eqref{idenmob3}, \eqref{secondsumpgreater5} and \eqref{psicount}, the first case of the corollary is established.
\item Assume that $p =3$. Then the condition that $t_i=0$ if $i\le 3$ and $3 \nmid i$ (appearing in the second term of \eqref{idenmob3}) is equivalent to $t_1=t_2=0$. This shows that the second term in \eqref{idenmob3} is 
\begin{equation}\label{secondsump3}
1_{t_1=t_2=0} \cdot \frac{1}{n} \sum_{3 \mid d \mid n, 9 \nmid d} \mu(d) \psi_q(\frac{n}{d},\frac{1}{d/3} t_3^{\frac{q}{3}}).
\end{equation}
From \eqref{idenmob3}, \eqref{secondsump3} and \eqref{eq:NewCarlitz}, the second case of the corollary is established.
\item Assume that $p=2$. Then the condition that $t_i=0$ if $i\le 3$ and $2 \nmid i$ (appearing in the second term of \eqref{idenmob3}) is equivalent to $t_1=t_3=0$. This shows that the second term in \eqref{idenmob3} is 
\begin{equation}\label{secondsump2}
1_{t_1=t_3=0} \cdot \frac{1}{n} \sum_{2 \mid d \mid n, 4\nmid d} \mu(d) \psi_q(\frac{n}{d},\frac{1}{d/2} t_2^{\frac{q}{2}}).
\end{equation}
From \eqref{idenmob3}, \eqref{secondsump2} and \eqref{eq:NewCarlitz}, we find that
\begin{equation}\label{idenmob3p2}
\begin{split}
\pi_q(n,t_1,t_2,t_3) &= \frac{1}{n} \sum\limits_{ 2 \nmid d \mid n} \mu(d) \psi_q( \frac{n}{d}, \frac{t_1}{d}, \frac{t_2}{d} + \frac{1-d}{2d^2} t_1^2, \frac{t_3}{d} + \frac{1-d}{d^2} t_1 t_2 + \frac{(2d-1)(d-1)}{6d^3}t_1^3) \\
& \quad +  1_{t_1=t_3=0}\cdot \frac{1}{n} \sum\limits_{ 2 \mid d \mid n} \mu(d) q^{\frac{n}{d}-1}.
\end{split}
\end{equation}
If $2 \nmid d \mid n$, then we have the following equalities in $\FF_q$:
\begin{equation}\label{odddcase}
\begin{split}
\frac{t_2}{d} + \frac{1-d}{2d^2} t_1^2 &= \begin{cases} t_2 & \text{if }d \equiv 1 \bmod 4, \\ t_2 + t_1^2 & \text{if }d \equiv 3 \bmod 4, \end{cases} \\
\frac{t_3}{d} + \frac{1-d}{d^2} t_1 t_2 + \frac{(2d-1)(d-1)}{6d^3}t_1^3 & = \begin{cases} t_3 & \text{if }d \equiv 1 \bmod 4, \\ t_3+ t_1^3 & \text{if }d \equiv 3 \bmod 4. \end{cases}
\end{split}
\end{equation}
The third case of the corollary follows from \eqref{idenmob3p2} and \eqref{odddcase}.
\end{enumerate}
\end{proof}

\bibliographystyle{alpha}
\bibliography{references}

Raymond and Beverly Sackler School of Mathematical Sciences, Tel Aviv University, Tel Aviv 69978, Israel.
E-mail address: ofir.goro@gmail.com

\end{document}